\documentclass[pdflatex,sn-mathphys-num]{sn-jnl}

\usepackage{graphicx}%
\usepackage{caption}
\usepackage{multirow}%
\usepackage{amsmath,amssymb,amsfonts}%
\usepackage{amsthm}%
\usepackage{mathrsfs}
\usepackage[title]{appendix}%
\usepackage{xcolor}%
\usepackage{textcomp}%
\usepackage{manyfoot}%
\usepackage{booktabs}%
\usepackage{algorithm}%
\usepackage{algorithmicx}%
\usepackage{algpseudocode}%
\usepackage{listings}%
\usepackage{makecell}

\usepackage[all,cmtip]{xy}

\usepackage{mathtools}
\usepackage{microtype}
\usepackage{latexsym}
\usepackage{yfonts}
\usepackage{graphics}
\usepackage{array}
\newcommand{\R}{\mathbb{R}}
\newcommand{\ed}{\text{\rm d}}
\newcommand{\eD}{\text{\rm D}}
\newcommand{\vol}{\text{\rm vol}}

\theoremstyle{thmstyleone}%
\theoremstyle{thmstyletwo}%
\newtheorem{example}{Example}%

\theoremstyle{thmstylethree}%

\newtheorem{dfn}{Definition}
\newtheorem{thm}{Theorem}
\newtheorem{prop}[dfn]{Proposition}
\newtheorem{lemma}[dfn]{Lemma}

\newtheorem{rem}{Remark}

\begin{document}

\title[The exterior derivative and the mean value equality in $\mathbb{R}^n$]{The exterior derivative and the mean value equality in $\mathbb{R}^n$}


\author[1]{\fnm{Daniel} \sur{Fadel}}\email{daniel.fadel@icmc.usp.br}

\author[2]{\fnm{Henrique} \sur{N. S\'a Earp}}\email{henrique.saearp@ime.unicamp.br}

\author*[2]{\fnm{Tom\'as} \sur{S. R. Silva}}\email{tomas@ime.unicamp.br}

\affil[1]{\orgdiv{Instituto de Ci\^encias Matem\'aticas e de Computaç\~ao (ICMC)}, \orgname{Universidade \indent de
S\~ao Paulo (USP)}, \orgaddress{\postcode{13566-590}, \city{S\~ao Carlos}, \state{S\~ao Paulo}, \country{Brazil}}}

\affil[2]{\orgdiv{Instituto de Matemática, Estatística e Computação Científica (IMECC)}, \orgname{Universidade Estadual de Campinas (UNICAMP)}, \orgaddress{ \postcode{13083-859}, \city{Campinas}, \state{S\~ao Paulo}, \country{Brazil}}}


\abstract{This survey revisits classical results in vector calculus and analysis by exploring a generalised perspective on the exterior derivative, interpreting it as a measure of ``infinitesimal flux''. This viewpoint leads to a higher-dimensional analogue of the Mean Value Theorem, valid for differential $k$-forms, and provides a natural formulation of Stokes’ theorem that mirrors the exact hypotheses of the Fundamental Theorem of Calculus -- without requiring full $C^1$ smoothness of the differential form.

As a numerical application, we propose an algorithm for exterior differentiation in $\mathbb{R}^n$ that relies solely on black-box access to the differential form, offering a practical tool for computation without the need for mesh discretization or explicit symbolic expressions.}

\keywords{Exterior derivative, Stokes' Theorem, Numerical Exterior Derivative}


\pacs[MSC Classification]{26B05, 58A10, 65D25}

\maketitle
\thanks{\itshape This article is dedicated to Professor Felipe Acker, on the occasion of his 70th birthday,\\ and to Professor Paolo Piccione, on the occasion of his 60th birthday.}

\newpage
\section{Introduction}
\label{sec:Introduction}
In this survey article, we revisit a perspective on the exterior derivative operator on differential forms that yields a higher-dimensional generalisation of the Mean Value Theorem (MVT) as an \emph{equality}, in contrast with the traditional presentation in vector calculus texts, where the theorem typically appears only in the form of an inequality. These ideas were first formulated by Felipe Acker
in his somewhat unknown articles  \cite{paper,paper2}. We propose here a thorough revision of that theory, clarifying some imprecisions in the original proofs and using it to deduce Stokes' theorem via more accessible and elementary arguments. While stronger formulations do exist -- most notably within the framework of modern \textit{Geometric Measure Theory} in terms of normal currents, cf. \cite[\S4]{Federer1996},
our aim is to present these results in an easily approachable language.


We begin with an alternative proof for the usual one-dimensional MVT (Theorem \ref{thm: one_dim}) and review a classical proof of the Fundamental Theorem of Calculus (FTC) that makes use of it (Theorem \ref{thm: ftc}). The cornerstone of this alternative approach is the following well-known property of continuous real-valued functions on closed intervals (proven in Section \ref{sec:trisection lemma}):
\begin{lemma}[Trisection Lemma]   \label{lemma: fundamental}
Let $f:[a,b]\longrightarrow\mathbb{R}$ be a continuous function (with $a<b$). Then, there exists a (proper) sub-interval $[a',b']\subset \;]a,b[$, satisfying the following:
\begin{enumerate}
\item 
$b'-a'=\frac{1}{3}(b-a)$;
\item 
$\displaystyle 
\frac{f(b')-f(a')}{b'-a'}
=\frac{f(b)-f(a)}{b-a}. $
\end{enumerate}
\end{lemma}

On the one hand, the usual textbook statement of the FTC does carry over
to $\R^n$ (and  manifolds, etc.) as Stokes' theorem, albeit at the cost of
tightening the hypotheses from the mere integrability of the derivative to actual $C^1$
regularity, so that the proof can be reformulated in terms of successive integration. On the other hand, the same is not true for the MVT, whose proof in higher dimensions is a repetition along paths of the one-dimensional version and yields at best an inequality. Our point is precisely that the Trisection Lemma \emph{does} generalise to higher dimension (see Lemma \ref{lemma: trisec for k-blocks}), and along with it a suitable statement of the MVT as an equality.

The conceptual key is a geometric interpretation of the exterior derivative as an `infinitesimal flux' (Definition \ref{dfn: flux}), much in the vein of the `physical'
interpretation of the divergence of a vector field (see, e.g., \cite[p.~189]{Arnold1989}). This alternative definition allows e.g. for the rigorous `differentiation' of certain discontinuous differential forms, and it coincides with the usual notion of exterior derivative
in the differentiable case (Theorem \ref{thm: compatibility}). As an illustration, consider the following lemma for $0$-forms in dimension $1$ (the proof is elementary and can be found at the end of \S\ref{subsec: compatibility}):

\begin{lemma}
\label{lemma: 0-forms dim 1}
Let $U\subset\R$  be an open subset and $f:U\to\R$ be a real function. Interpreting any closed interval $[a,b]\subseteq U$ as an \emph{oriented $1$-chain with
boundary} (see Section \ref{subsec: chains-blocks-integration} for definitions), we denote 
$$
\vol[a,b]:=b-a
\quad\text{and}\quad
\int_{\partial [a,b]}f:=f(b)-f(a). 
$$
Then
$$
f\text{ is differentiable at }x\in U,\text{ with }f'(x)=L\iff \exists\lim_{[a,b]\to x}\frac{1}{\vol[a,b]}\int_{\partial [a,b]}f =L,
$$ 
\textit{where the last condition means: $\forall\varepsilon>0$, $\exists\delta>0$ such that} 
$$
[a,b]\subseteq U, x\in[a,b], \vol[a,b]<\delta \Rightarrow \left|\frac{f(b)-f(a)}{b-a} - L\right|<\varepsilon.
$$

\end{lemma}

\noindent Then the geometrical interpretation of the derivative of $f$ at the point $x$ is a measure of the flux of $f$ through the boundary of small intervals containing $x$.

In the general setting of $(k-1)$-forms in dimension $n$, we mimic Lemma \ref{lemma: 0-forms dim 1} and define the exterior derivative $\eD\omega_x$ of $\omega$ at the point $x$ by requiring that, for every choice of directions $v_1,\dots,v_k \in \R^n$, the value of $(\eD\omega)_x(v_1,\dots,v_k)$ be given by the limit of the normalised flux of $\omega$ across the boundary of $C^1$-parametrised $k$-blocks shrinking to $x$ (see Definition \ref{dfn: flux}). In dimensions $n>1$ and degree $k>1$, the right-hand side of this definition still makes sense even when $\omega$ is not continuous at $x$, let alone differentiable (see e.g. Example \ref{ex:constant-null}). Moreover, we prove that in the differentiable case this flux-based exterior derivative $\eD\omega$ agrees with the usual exterior derivative $\ed\omega$ (see Theorem \ref{thm: compatibility}). Thus, the definition extends the classical operator $\ed$ to a strictly larger class of differential forms. Moreover, with this viewpoint, we can prove the following generalisation of MVT as an equality (see Theorem \ref{thm: mvt}):

\begin{thm}[MVT for differential forms]
Let $\omega$ be a $(k-1)$-form defined on a $k$-block $B=[a_1,b_1]\times\ldots\times[a_k,b_k]\subset\mathbb{R}^k$ such that:
\begin{enumerate}
\item $\omega$ is continuous\footnote{In Theorem \ref{thm: mvt} we ask for a slightly weaker continuity-type assumption on $\omega$; see Definition \ref{dfn: cont}.};
\item $\omega$ is \emph{derivable} on $\mathring{B}$, i.e. $\omega$ admits an exterior derivative $\eD\omega$ in the broader sense of Definition~\ref{dfn: flux}, which agrees with the usual $\ed\omega$ when $\omega$ is differentiable (Theorem \ref{thm: compatibility});
\end{enumerate} 
then there exists $\xi\in\mathring{B}$ such that 
$$
{\eD\omega}_\xi(e_1,\ldots,e_k) = \frac{1}{\mathrm{vol}(B)}\int_{\partial{B}}\omega.
$$
\end{thm}

As a result, the original spirit of the one-dimensional proof of the FTC using the MVT can be emulated to show an elementary version of Stokes’ theorem (see Theorem \ref{thm: diff_stokes}), with the typical $C^1$ hypothesis on the integrand relaxed:

\begin{thm}[Stokes without $C^1$ assumption]
\label{thm: Stokes without C1}
Let $\omega\in\Omega^k(U)$ be a $k$-form on some open subset $U\subset\mathbb{R}^n$, and let $c:B\rightarrow U$ be a $(k+1)$-singular block of class $C^1$ in $U$ (see Section \ref{subsec: chains-blocks-integration}). If $\omega$ is continuous on $c(B)$ and differentiable on $c(\mathring{B})$, then
$$ 
\int_{\partial c} \omega = \int_{c} \ed\omega, 
$$ whenever the right-hand side integral exists. More generally, let $c=n_1c_1+\ldots+n_lc_l$ be a $(k+1)$-chain of class $C^1$, with $c_i:B_i\rightarrow U$, $i=1,\dots,l$. If $\omega$ is continuous on $c(B_i)$ and differentiable on $c(\mathring{B_i})$, for each $i=1,\ldots,l$, then the same conclusion holds.
\end{thm}

This version of Stokes' theorem has precisely the analogous hypotheses of the one-dimensional FTC (see Theorem \ref{thm: ftc}): continuity on the domain block $[a,b]$, differentiability on the image of its interior (as we have the inclusion $1$-singular block $[a,b]\hookrightarrow\mathbb{R}$, the hypothesis translates to differentiability on the interior $]a,b[$) and Riemann integrability of the right-hand side.

The text is intended to be accessible to an advanced undergraduate student, but we believe that even experienced mathematicians may find our construction intriguing and even useful for teaching. Furthermore, Section~\ref{sec:numerical-exterior-derivative} presents a contemporary numerical application of the theory: a fully-functioning sample-based implementation of the exterior derivative on~$\mathbb{R}^n$, based on its flux interpretation (Definition~\ref{dfn: flux}), which to our knowledge is completely original in this perspective and could have direct uses in computational calculus on surfaces. 

\medskip

\noindent\textbf{Organisation of the paper.}
Section~\ref{sec:trisection lemma} develops the subdivision framework on singular chains that makes it possible to extend to higher dimensions the one-dimensional subdivision argument underlying the Fundamental Theorem of Calculus. The Trisection Lemma and its higher-dimensional extensions provide a systematic decomposition of chains into controlled pieces, yielding in particular the notion of flux-continuity. These results furnish the geometric control on boundary integrals that is essential for the limiting arguments used throughout the paper.

Section~\ref{sec: Exterior derivative} introduces the exterior derivative as an infinitesimal flux. Following the paradigm outlined in the introduction, we define the operator $\eD$ via normalised boundary integrals over shrinking blocks and establish its fundamental properties, including a mean value equality for differential forms, a Stokes formula for $\eD$, the identity $\eD^2=0$, and compatibility with the classical exterior derivative in the differentiable setting.

Section~\ref{sec:Stokes_Theorem} combines these results to obtain a Stokes theorem for the classical exterior derivative under hypotheses directly analogous to those of the Fundamental Theorem of Calculus.

Finally, Section~\ref{sec:numerical-exterior-derivative} presents a numerical application of the theory. The flux interpretation leads to a natural and easily implementable approximation scheme for the exterior derivative based on boundary integrals, illustrating that this viewpoint is not only conceptually clarifying but also computationally effective.




\section{Repercussions of the Trisection Lemma}
\label{sec:trisection lemma}                
We will discuss a rather non-standard proof (cf. \cite[p.~160]{acker1} and \cite[p.~210]{paper}) of the one-dimensional Mean Value Theorem. The essence of all generalisations we carry out in this survey lies in the Trisection
Lemma stated in the Introduction (Lemma \ref{lemma: fundamental}), so let us begin with its proof.

\begin{proof}[Proof of Lemma \ref{lemma: fundamental}]
Take the partition $P=\{a\eqqcolon a_0<a_1<a_2<a_3\coloneqq b\}$ of $[a,b]$ with ${a_i-a_{i-1} = \frac{1}{3}(b-a)\eqqcolon h}$, for each $i=1,\dots,3$, so that
$$ 
\alpha\coloneqq\frac{f(b)-f(a)}{b-a}
=\frac{1}{3}\sum\limits_{i=1}^3 \frac{f(a_i)-f(a_{i-1})}{a_i-a_{i-1}}.
$$ 
If the three terms of the sum on the right-hand side  are equal, we just take $b'\coloneqq a_2$ and $a'\coloneqq a_1$. Otherwise the continuous function $m:[a,b-h]\longrightarrow\mathbb{R}$, defined by $$ m(x)=\frac{f(x+h)-f(x)}{h} - \alpha, $$ assumes positive and negative values on two of the $a_i$'s, for $i=0,1$ or $2$. Then one can restrict $m$ to the interval $I$, say, determined
by such points and invoke the intermediate value theorem, given the connectedness of $I$ and continuity of $m$. This ensures the existence of some interior $c\in I\subset \;]a,b-h[$ such that $m(c)=0$; in this case, take $b'\coloneqq c+h$ and $a'=c$.                          
\end{proof}

\subsection{The mean value theorem and the fundamental theorem
of calculus}

Let us recall the statement of the one-dimensional version of the MVT (outlined in \cite[p.~160]{acker1} and \cite[p.~210]{paper})
and its proof using the Trisection Lemma:

\begin{thm}[M.V.T., $1$-dimensional case]
\label{thm: one_dim}
If $f:[a,b]\longrightarrow\mathbb{R}$ is a continuous function in $[a,b]$ and differentiable in $]a,b[$, then 
$$ \exists c\in\;]a,b[ \quad\text{such that}\quad f'(c)=\frac{f(b)-f(a)}{b-a}. $$
\end{thm}
\begin{rem}
    Lemma \ref{lemma: 0-forms dim 1}, whose proof can be found at the end of \S\ref{subsec: compatibility}, shows that in the context of $0$-forms in dimension $1$, differentiability is equivalent to being \emph{derivable} in the sense of Definition~\ref{dfn: flux}.
\end{rem}
\begin{proof}
Noting that continuity is preserved by restriction, and that the differentiability of $f$ on $]a,b[$ implies the differentiability of $f|_{[c,d]}$ on $]c,d[$, for every $[c,d]\subset[a,b]$, we can iterate the last Lemma to obtain a descending chain of compact sets $[a_i,b_i]$, 
$$ ]a,b[\;\supset[a_1,b_1]\supset\ldots\supset[a_n,b_n]\supset\ldots, $$
with the following properties,  for every $n\in\mathbb{N}$:
\begin{enumerate}
\item[(1)] 
$\displaystyle
b_n-a_n=\frac{1}{3^n}(b-a)$;  
\item[(2)] 
$\displaystyle 
\frac{f(b_n)-f(a_n)}{b_n-a_n}
=\frac{f(b)-f(a)}{b-a}.$
\end{enumerate}

Now, as a decreasing nested sequence of non-empty compact sets in the Hausdorff compact space $[a,b]$, one has $X\coloneqq\bigcap [a_n,b_n]\neq\emptyset$. Moreover,
property $(1)$ ensures that $X$ consists of a single point, say $c$. We claim that 
$$ 
f'(c) 
= 
\lim_{n\to\infty} \frac{f(b_n)-f(a_n)}{b_n-a_n} 
= 
\frac{f(b)-f(a)}{b-a}.
$$ In fact, the second equality follows from (2), so we are left to prove the first one. We make use of the characterisation given by Lemma \ref{lemma: 0-forms dim 1}. Let $\varepsilon>0$ be given. The differentiability of $f$ in $]a,b[$ and the fact that $\{c\}=X\subset[a_1,b_1]\subset\;]a,b[$
together imply the differentiability of $f$ at $c$, thus Lemma \ref{lemma: 0-forms dim 1} guarantees that there exists $\delta>0$ such that 
$$
[x,y]\subseteq]a,b[, c\in[x,y], y-x<\delta \Rightarrow \left|\frac{f(y)-f(x)}{y-x} - f'(c)\right|<\varepsilon.
$$ Now, by (1) there exists $n_0\in\mathbb{N}$ such that $b_n-a_n<\delta$ for every $n\geq n_0$. Since $c\in [a_n,b_n]\subseteq\;]a,b[$ for each $n\in\mathbb{N}$, we have
$$
n\geq n_0 \Rightarrow \left|\frac{f(b_n)-f(a_n)}{b_n-a_n} - f'(c)\right|<\varepsilon.
$$                                        This completes the proof.
\end{proof}

We next revisit the classical proof of the Second Fundamental Theorem of Calculus, highlighting the perspective that will inform our generalisation of the MVT

\begin{thm}[Second Fundamental Theorem of Calculus]
\label{thm: ftc}
If $f:[a,b]\longrightarrow\mathbb{R}$ is continuous on $[a,b]$, differentiable on $]a,b[$ and $f'$ is integrable on $[a,b]$, then 
$$ \int_{\partial[a,b]} f = \int_{[a,b]} f'. $$
\end{thm}

\begin{proof}
Let $P=\{a=a_0<\ldots<a_n=b\}$ be a partition of $[a,b]$. We denote by $s(f';P)$ and $S(f';P)$, respectively, the lower and upper sums of $f'$ relative to $P$. Writing $$\int_{\partial[a,b]}f=f(b)-f(a)=\sum\limits_{i=1}^n {f(a_i)-f(a_{i-1})},$$ we apply the MVT to each restriction  $f|_{[a_{i-1},a_i]}$, $i=1,\ldots,n$. Then there exist $\xi_i\in\;]a_{i-1},a_i[$ such that $f(a_i)-f(a_{i-1})=f'(\xi_i)(x_i-x_{i-1})$, so $$\int_{\partial[a,b]}f=\sum\limits_{i=1}^n {f'(\xi_i)(x_i-x_{i-1})},$$ which implies $s(f';P)\leq{\int_{\partial[a,b]}f}\leq{S(f';P)}$.
This yields the claim, by the generality of $P$ and integrability of $f'$ on $[a,b]$.
\end{proof}

The previous proof illustrates the importance of the MVT equality form to show Theorem \ref{thm: ftc}. Moreover, it suggests that, if a similar version of the MVT were available for differential forms in higher dimensions, then one might expect to prove some version of Stokes' theorem with similar assumptions via some kind of mean value argument. This is precisely the strategy we shall follow, but that will require some rigorous vocabulary for integration in $\R^n$.

\subsection{Chains, blocks and integration in \texorpdfstring{$\R^n$}{Rn}}\label{subsec: chains-blocks-integration}

If $f:A\rightarrow W$ is a function between open subsets $A\subseteq\mathbb{R}^k$ and $W\subseteq\mathbb{R}^n$, we know $f$ is of class $C^l$ ($l\geq 1$) if all the partial derivatives of $f$ up to order $l$ exist and are continuous on $A$. More generally, if $X\subseteq\mathbb{R}^k$ and $Y\subseteq\mathbb{R}^n$ are arbitrary subsets, a function $f:X\rightarrow Y$, is said to be of {\textbf{class} {\boldmath $C^l$}} if, for every point $x\in X$, there exists some open neighbourhood $U$ of $x$ and a $C^l$ function $F:U\rightarrow\mathbb{R}^n$ such that $f|_{U\cap X}=F$.

\begin{dfn}[$k$-blocks, singular $k$-blocks and $k$-chains]
A {\textbf{{\boldmath$k$}-block}} $B=B^k$ in $\mathbb{R}^k$ ($k\geqslant 1$) is a $k$-product of non-degenerate closed intervals of the real line, i.e., $B^k:=\prod_{i=1}^k [a_i,b_i]$, where $\{a_i\},\{b_i\}\neq[a_i,b_i]\subseteq\mathbb{R}$, $i=1,\ldots, k$. Given an open subset $U\subset\mathbb{R}^n$, a \textbf{singular {\boldmath$k$}-block of class $C^l$ ($l\geq 1$) in $U$} is a map of the form $c:B^k\rightarrow U$ of class $C^l$. A \textbf{{\boldmath$k$}-chain (of class $C^l$) in $U$} is a \emph{formal} (finite) sum $c = n_1c_1+\ldots+n_pc_p$ of singular $k$-blocks $c_i$ (of class $C^l$) in $U$, with integer coefficients $n_i\in\mathbb{Z}$, $i=1,\ldots,p$. Finally, a \textbf{(singular) {\boldmath$0$}-block in $U$} is simply a point in $U$ (a function $c:\{0\}\to U$).
\end{dfn}

\begin{dfn}[Boundary of a singular $k$-block]
\label{dfn:boundary}
Let $k\geq 1$ and $B\coloneqq\prod\limits_{i=1}^k[a_{i0},a_{i1}]$ be a $k$-block in $\mathbb{R}^k$ and $c$ a singular $k$-block in $\mathbb{R}^n$ defined on $B$. For each $i\in\{1,\ldots,k\}$ and $j\in\{0,1\}$, define 
$$ c_{ij}:\prod\limits_{l\neq i}[a_{l0},a_{l1}]\longrightarrow\mathbb{R}^n $$ by 
$$ c_{ij}(t_1,\ldots,\hat{t_i},\ldots,t_k)\coloneqq c(t_1,\ldots,a_{ij},\ldots,t_k). $$ Then, the \textbf{{boundary of $c$}} is the $(k-1)$-chain $\partial c$ defined by

$$ \partial c\coloneqq \sum\limits_{i=1}^k\sum\limits_{j=0}^1(-1)^{i+j}c_{ij}. $$ If $c=n_1c_1+\ldots n_pc_p$ is a $k$-chain, $k\geq 1$, the \textbf{{boundary of $c$}} is the $(k-1)$-chain defined by $$\partial c\coloneqq n_1\partial c_1+\ldots+n_p\partial c_p.$$ 

Finally, we define the boundary of a $0$-chain to be the zero map.
\end{dfn}
\begin{rem}
In Definition \ref{dfn:boundary}, suppose that $n=k$ and that 
$B=\prod\limits_{i=1}^k [a_{i0},a_{i1}] \subset U$ for some open set 
$U \subset \mathbb{R}^k$. In this situation, we regard $B$ as the singular 
$k$-block in $U$ determined by the inclusion map $c:B \hookrightarrow U$. 
We then write $\partial B$ for the underlying set of $\partial c$, and in this 
case the maps $c_{ij}$ are precisely the faces of the $(k-1)$-block $\partial B$.
\end{rem}

An important property of the boundary operation $\partial$ is that $\partial(\partial c)=0$, for every $k$-chain $c$ on some open subset $A$ of $\mathbb{R}^n$.\label{page boundary2=0} The interested reader can consult \cite[Theorem 4-12, p.99]{spivak} for a
proof of this fact.

We denote by $\Omega^k(U)$ the set of all differential $k$-forms defined on some open set $U\subset\mathbb{R}^n$, i.e., all the functions of the form $U\rightarrow\Lambda^k (\mathbb{R}^n)^{\ast}$ (no regularity assumptions). 

\begin{dfn}[Integration on blocks and chains]\label{dfn:integration}
If $\omega$ is a $k$-form defined on a $k$-block $B$ in $\mathbb{R}^k$, then $\omega=f {\ed}x^1\wedge\ldots\wedge {\ed}x^k$ for a unique function $f:B\rightarrow\mathbb{R}$, where $(x^1,\ldots,x^k)$ are the canonical coordinates on $\mathbb{R}^k$, and $\ed x^i$ is the differential of $x^i$. We define the {\textbf{integral of $\omega$ over $B$}} to be
$$
\int_{B}\omega \coloneqq \int_{B}f,
$$ when the right-hand side (Riemann) integral exists. More generally, if $\omega\in\Omega^k(U)$, for an open subset $U$ of $\mathbb{R}^n$ containing the image of a singular $k$-block $c:B^k\rightarrow U$ of class $C^1$, we define the {\textbf{integral of $\omega$ over $c$}} to be $$ \int_{c} \omega \coloneqq \int_{B^k} c^{\ast}\omega, $$ when the right-hand side exists. When $c$ is a $k$-chain in $U$ of the form $c=n_1c_1+\ldots+n_pc_p$, we define 
$$
\int_{c}\omega \coloneqq n_1\int_{c_1}\omega + \ldots + n_p\int_{c_p}\omega,
$$ when each of the integrals on the right-hand side exists.
\end{dfn}
\begin{rem}
    Using Definitions \ref{dfn:boundary} and \ref{dfn:integration}, it is easy to check that if $\omega\in\Omega^{k-1}(U)$, where $U\subset\mathbb{R}^n$ is an open subset, and $c:B\to U$ is a singular $k$-block of class $C^1$, then 
    \[
    \int_{\partial c}\omega = \int_{\partial B} c^{\ast}\omega,
    \] provided the (Riemann) integrals exist.
\end{rem}

We now try to sketch a proof of the following heuristic statement: 
\begin{quote}
\emph{If $\omega$ is a continuous $(k-1)$-form defined on a $k$-block $B$, derivable  in the interior $\mathring{B}$ of this block (i.e. admitting an exterior derivative in a sense to be made precise -- see Definition~\ref{dfn: flux}), and such that $\ed\omega$ is an integrable $k$-form on $B$, then }$$\int_{\partial B}\omega=\int_{B}\ed\omega.$$
\end{quote}

\noindent Imitating the proof of Theorem \ref{thm: ftc}, take a partition $P$ of $B$ and consider $S(P)=\{\beta_1,\ldots,\beta_n\}$ the set of sub-blocks of $B$ determined by $P$. Thus, $$\int_{\partial B} \omega = \sum\limits_{\beta\in S(P)} {\int_{\partial\beta} \omega}.$$ Then, we will be able to establish the statement if, in the interior of each block $\beta\in S(P)$, there exists some $\xi_\beta$ such that 
$$
\int_{\partial\beta} \omega = \text{vol}(\beta)\ed\omega_{\xi_\beta}(e_1,\ldots,e_k),
$$ where here and henceforth $e_i$ denotes the $i$-th element of the canonical basis of $\mathbb{R}^k$.

This will be the generalisation we are seeking for the Mean Value Theorem  in \textit{equality form}. Once we have such improved version at hand, a suitable version of the second fundamental theorem of calculus will follow as an easy consequence of the preceding argument.

\subsection{Generalisation of the Trisection Lemma} 

Now we know in what sense to expect a generalisation of the MVT, the first concrete task is to extend Lemma   \ref{lemma: fundamental}  (see \cite[p.~274]{acker1} and \cite[p.5]{paper2} for more informal proofs, and \cite[p.~213]{paper} for a two-dimensional version).                
\begin{lemma}[Trisection Lemma for $k$--blocks]
\label{lemma: trisec for k-blocks}
\label{thm: lema}
Let $\omega$ be a continuous $(k-1)$-form defined on a $k$-block $B=[a_1,b_1]\times\ldots\times[a_k,b_k]\subseteq\mathbb{R}^k$. Then, there exists a $k$-block $B_1\subset\mathring{B}$ such that:
\begin{enumerate}
\item   
the sides of $B_1$ measure $1/3$ of the sides of $B$;
\item   
$\displaystyle
\frac{1}{\mathrm{vol}(B_1)}\int_{\partial{B_1}}\omega
= 
\frac{1}{\mathrm{vol}(B)}\int_{\partial{B}}\omega.$
\end{enumerate}
\end{lemma}

\begin{proof}
Indeed, adapting the proof of the one-dimensional case, we divide $B$ in $3^k$ blocks, say $\beta_1,\ldots,\beta_{3^k}$, whose $i$-th side has length  $h_i\coloneqq 1/3(b_i-a_i)$, for each $i=1,\ldots,k$,
so the integral decomposes as $$\int_{\partial B} \omega = \sum\limits_{i=1}^{3^k} \int_{\partial\beta_i} \omega.$$ Since $\text{vol}(B) = 3^k\text{vol}(\beta_i), \forall i\in\{1,\ldots,{3^k}\}$, equality goes over to the averages:
$$
\frac{3^k}{\text{vol}(B)}\int_{\partial{B}}\omega 
= 
\sum\limits_{i=1}^{3^k} \frac{1}{\text{vol}(\beta_i)}\int_{\partial\beta_i} \omega,
$$
or, equivalently, 
\begin{equation}\label{eq: m}
\sum\limits_{i=1}^{3^k}\left(\frac{1}{\text{vol}(\beta_i)}{\int_{\partial\beta_i} \omega} - \frac{1}{\text{vol}(B)}\int_{\partial{B}}\omega\right) = 0.
\end{equation}
Let $h\coloneqq (h_1,\ldots,h_k)\in\mathbb{R}^k$ and $a\coloneqq (a_1,\ldots,a_k)\in B$. Put $\beta\coloneqq[0,h_1]\times\ldots\times[0,h_k]$ and $$m: [a_1,b_1-h_1]\times\ldots\times[a_k,b_k-h_k]\rightarrow\mathbb{R}$$ defined by $$m(x)\coloneqq\frac{1}{\text{vol}(\beta)}\int_{\partial(x+\beta)}{\omega} - \frac{1}{\text{vol}(B)}\int_{\partial B} \omega,$$ where $x+\beta\coloneqq\{x+y:y\in\beta\}$. {}Now, either $(i)$ all the terms of the sum (\ref{eq: m}) are equal to zero, or $(ii)$ there exists at least one which is greater than zero and another one smaller than zero, so the function $m$ assumes negative and positive values in its domain. In case $(i)$, we will certainly have $m(a+h)=0$, and we can just take $B_1:=(a+h)+\beta\subset\mathring{B}$. For $(ii)$, let $x_0,x_1\in[a_1,b_1-h_1]\times\ldots\times[a_k,b_k-h_k]$ such that $m(x_0)<0$ and $m(x_1)>0$. Then we can define $\gamma:[0,1]\rightarrow B$ as the concatenation of the straight lines from $x_0$ to $x_{\ast}\coloneqq a+h$ and from $x_{\ast}$ to $x_1$. Then, $m\circ\gamma:[0,1]\rightarrow\mathbb{R}$ is continuous (as a composition of continuous functions) and, by construction, $m\circ\gamma(0)<0$ and $m\circ\gamma(1)>0$. Therefore, by the intermediate value theorem, there exists $\theta\in\;]0,1[$  such that $m\circ\gamma(\theta)=0$. Set $x_0\coloneqq\gamma(\theta)$ and define $B_1:=x_0+\beta\subset\mathring{B}$.
\end{proof}

Note in the above proof that we only needed the continuity of $\omega$ to ensure the continuity of $m$. Therefore, the relevant
continuity is not that of $\omega$ in the usual sense, but a weaker one concerning
the pairing between
differential forms and chains given by the integral. Given a $(k-1)$-form $\omega$, we want to ensure that the integral of $\omega$ on the boundary of  nearby $k$-blocks takes nearby values. This motivates the following definition (see \cite[p.~273]{acker1}; also compare with \cite[p.~213]{paper} and \cite[p.4]{paper2}).

\begin{dfn}[Flux-continuity]\label{dfn: cont}
A $(k-1)$-form $\omega$ defined on a $k$-block $B$ in $\mathbb{R}^k$ is {\textbf{flux-continuous}} on $B$ if, $\forall\beta=[0,h_1]\times\ldots\times[0,h_k]$, with $h_i<(b_i-a_i), i=1,\ldots,k$, the real function $$ x\mapsto\int_{\partial(x+\beta)}\omega $$ is continuous at every point $x\in B$ such that $x+\beta\subseteq B$. Here, as usual, $x+\beta:=\{x+y:y\in\beta\}$.
\end{dfn}

Thus, what we have actually proved in the proof of Lemma \ref{lemma: trisec for k-blocks} is the following (cf. \cite[p.~274]{acker1}):

\begin{lemma}[Trisection Lemma under flux-continuity]
\label{thm: fundlemma}
Let $\omega$ be a \emph{flux-continuous} $(k-1)$-form defined on a $k$-block $B=[a_1,b_1]\times\ldots\times[a_k,b_k]\subseteq\mathbb{R}^k$. Then, there exists a $k$-block $B_1\subset\mathring{B}$ such that:
\begin{enumerate}
\item   the sides of $B_1$ measure $1/3$ of the sides of $B$;
\item   $\displaystyle\frac{1}{\mathrm{vol}(B_1)}\int_{\partial{B_1}}\omega = \frac{1}{\mathrm{vol}(B)}\int_{\partial{B}}\omega.$
\end{enumerate}
\end{lemma}

At least for $k>1$, it's clear that this version of the lemma now includes e.g. those $(k-1)$-forms whose component functions may be discontinuous on a countable subset of its domain $k$-block. In fact, it is easy to produce more interesting examples illustrating that the flux-continuity condition is way weaker than the usual continuity for the $(k-1)$-form.

\begin{example}
   Let $B=\prod_{i=1}^k[a_i,b_i]\subseteq\mathbb{R}^k$ be a $k$-block, fix a continuous function $F:\,B\to\mathbb R$ and let
\[ S:\prod_{i=2}^k[a_i,b_i]\to\mathbb R \]
be any bounded function which has at least one point of discontinuity. Define the function
\[ A(x_1,x_2,\dots,x_k):=F(x_1,x_2,\dots,x_k)+S(x_2,\dots,x_k),\qquad x=(x_1,\ldots,x_k)\in B,\]
and consider the $(k-1)$-form
\begin{equation}\label{eq:omega-general}
\omega(x):=A(x)\,{\ed}x_2\wedge\cdots\wedge {\ed}x_k.
\end{equation}

The coefficient $A$ is continuous in the $x_1$ variable (because $F$ is), but may be discontinuous in the remaining variables (because of $S$). In fact, assume $S$ is discontinuous at some point $(y_2^0,\dots,y_k^0)\in\prod_{i=2}^k[a_i,b_i]$. Let $p=(p_1,y_2^0,\dots,y_k^0)\in B$ be any point with that projection. We are going to show that the $(k-1)$-form $\omega$ defined in \eqref{eq:omega-general} is not continuous at $p$, but $\omega$ is flux-continuous on $B$.

\emph{(Non-continuity at $p$)} Since $S$ is discontinuous at $(y_2^0,\dots,y_k^0)$ there exists a sequence $y^{(n)}=(y_2^{(n)},\dots,y_k^{(n)})\to(y_2^0,\dots,y_k^0)$ for which $S(y^{(n)})\to s_1$ and another sequence $z^{(n)}\to(y_2^0,\dots,y_k^0)$ for which $S(z^{(n)})\to s_2$ with $s_1\ne s_2$. Now using the first component $p_1$ of $p$, set $x^{(n)}:=(p_1,y_2^{(n)},\dots,y_k^{(n)})$ and $w^{(n)}:=(p_1,z_2^{(n)},\dots,z_k^{(n)})$. Then
\[
A(x^{(n)})=F(x^{(n)})+S(y^{(n)})\to F(p)+s_1,\qquad A(w^{(n)})=F(w^{(n)})+S(z^{(n)})\to F(p)+s_2,
\]
and consequently the coefficient function $A$ (hence the form $\omega$) does not admit a single limit at $p$. Thus $\omega$ is not continuous at $p$.

\emph{(Flux-continuity on $B$)} Fix $h=(h_1,\dots,h_k)$ with $0<h_i<b_i-a_i$ and let $x=(x_1,\dots,x_k)\in B$ be such that $x+\beta_h\subset B$, where $\beta_h:=[0,h_1]\times\ldots\times[0,h_k]$. We compute the integral of $\omega$ on the boundary of the block $x+\beta_h$. The boundary is the union of the $2k$ faces
\[ F_i^-:=\{y\in x+\beta_h: y_i=x_i\},\qquad F_i^+:=\{y\in x+\beta_h: y_i=x_i+h_i\},\qquad i=1,\dots,k,\]
each endowed with the outward orientation.

We claim that for each $i\ge 2$ both integrals of $\omega$ on $F_i^-$ and $F_i^+$ are zero. Indeed, if $i\ge 2$ then the inclusion map of the face $F_i^{\pm}$ into $\mathbb R^k$ pulls back the $1$-form ${\ed}x_i$ to the zero $1$-form on the face (because the coordinate $x_i$ is constant on the face). Since
\[ \omega=A(x)\,{\ed}x_2\wedge\cdots\wedge {\ed}x_k\]
contains the factor ${\ed}x_i$ for every $i\ge 2$, its pullback to $F_i^{\pm}$ is zero and therefore
\[ \int_{F_i^{\pm}}\omega=0\quad\text{for }i\ge 2.\]
Consequently only the two faces $F_1^-$ and $F_1^+$ may contribute to the boundary integral. On these faces the coordinates are naturally $(y_2,\dots,y_k)\in\prod_{i=2}^k[x_i,x_i+h_i]$ and the pullback of ${\ed}x_2\wedge\cdots\wedge {\ed}x_k$ is the standard volume form $dy_2\wedge\cdots\wedge dy_k$ on the face. Taking into account the outward orientation (the face $F_1^+$ has outward normal $+e_1$ and $F_1^-$ has outward normal $-e_1$) we obtain
\begin{align*}
\int_{\partial(x+\beta_h)}\omega
&=\int_{F_1^+}A(y)\,dy_2\cdots dy_k -\int_{F_1^-}A(y)\,dy_2\cdots dy_k\\
&=\int_{\prod_{i=2}^k[x_i,x_i+h_i]}\Big[A(x_1+h_1,y_2,\dots,y_k)-A(x_1,y_2,\dots,y_k)\Big]\,dy_2\cdots dy_k.
\end{align*}

Since $A=F+S$ and $S$ does not depend on $x_1$ the integrand simplifies to
\[ A(x_1+h_1,y_2,\dots,y_k)-A(x_1,y_2,\dots,y_k) = F(x_1+h_1,y_2,\dots,y_k)-F(x_1,y_2,\dots,y_k).\]
As $F$ is continuous on $B$, the function
\[ x\longmapsto F(x_1+h_1,x_2,\dots,x_k)-F(x_1,x_2,\dots,x_k) \]
is continuous on the set of admissible $x$, and therefore the integral above depends continuously on $x$. Thus $I_h(x)=\int_{\partial(x+\beta_h)}\omega$ is continuous in $x$ for every admissible $h$, which is precisely the statement that $\omega$ is flux-continuous.

To give an explicit $2$-dimensional example, let $k=2$, $B=[-1,1]\times[-1,1]\subset\mathbb R^2$, and define
\[ \omega(x,y):=(\sin x + H(y))\,dy, \]
where $H(y)$ is the Heaviside step function
\[ H(y)=\begin{cases}0,&y<0,\\1,&y\ge0.\end{cases} \]
The form $\omega$ is not continuous along the horizontal line $y=0$ because of the jump of $H$. Indeed, the coefficient of $dy$ has distinct one-sided limits at any point $(x,0)$. Nevertheless, fixing any $\beta_h=[0,h_1]\times[0,h_2]$ and any translation $x=(x_1,x_2)$ with $x+\beta_h\subset B$, a direct computation of the boundary integral gives
\[ \int_{\partial(x+\beta_h)}\omega = h_2\big(\sin(x_1+h_1)-\sin x_1\big). \]
(Only the two vertical sides contribute; the horizontal sides lie in directions tangent to $dy$ and do not contribute.) The right-hand side is smooth in $x$, hence continuous, so $\omega$ is flux-continuous on $B$ although it fails to be continuous along $y=0$.

These examples show the mechanism behind the failure of a flux-continuous form to be continuous: flux-continuity tests continuity of certain finite differences of the coefficients of a form (those finite differences which appear as differences of the coefficient along the face-pair for one coordinate). Any discontinuous dependence in coordinates orthogonal to that distinguished direction will cancel in the boundary integral and so will not be detected by the flux-continuity condition.

One can vary the construction by choosing different distinguished coordinates (replace the factor ${\ed}x_2\wedge\cdots\wedge {\ed}x_k$ by a wedge product that omits some other ${\ed}x_j$) or by making $S$ discontinuous on a set of positive measure, or even dense sets; the same cancellation effect persists.
\end{example}

\section{The exterior derivative as an infinitesimal flux}
\label{sec: Exterior derivative}

We recall some standard definitions to fix notation. Let $\omega\in\Omega^{k-1}(U)$, where $U\subset\R^n$ is an open subset, and write 
\begin{equation}
\label{eq:general_form}
    \omega = \sum_I \omega_I \ed x^I,
\end{equation} where the sum is taken over all increasing multi-indices $I\coloneqq\{i_1<\ldots<i_{k-1}\}\subseteq\{1,\ldots,n\}$, we have $\ed x^I:=\ed x^{i_1}\wedge\ldots\wedge \ed x^{i_{k-1}}$, and the component functions $\omega_I:U\rightarrow\mathbb{R}$ are given by
\[
\omega_I(x)=\omega_x(\partial_{i_1},\ldots,\partial_{i_{k-1}}).
\] Here $(x^1,\ldots,x^n)$ are the canonical coordinates on $\mathbb{R}^n$. If $x\in U$, we say $\omega$ is \textbf{continuous at $x$} if the component functions $\omega_I$ are continuous at $x$. We could define differentiability
analogously, replacing ``continuous" by ``differentiable". However, for our purposes, it will be more convenient to say $\omega$ is \textbf{differentiable} at $x$ when there exists a linear map
$$
\omega'(x):\mathbb{R}^n\rightarrow\Lambda^{k-1}(\mathbb{R}^n)^{\ast},
$$ 
such that 
$$
\lim_{h\to 0}\frac{\left\Vert\omega(x+h)-\omega(x)-\omega'(x)h\right\Vert}{|h|}=0, 
$$ 
where 
$\left\Vert\eta\right\Vert\coloneqq \text{sup}\{|\eta(v_1,\ldots,v_{k-1})|:|v_1|=\ldots=|v_{k-1}|=1\}$
is the standard operator norm\footnote{$\left\Vert\cdot{}\right\Vert$ could be any norm: they are all equivalent on a finite-dimensional real vector space. We made that particular choice for some minor later convenience.} for $\eta\in\Lambda^{k-1}(\mathbb{R}^n)^{\ast}$,
so that
\begin{equation}\label{eq: des1}
|\eta(v_1,\ldots,v_{k-1})|\leq\left\Vert\eta\right\Vert|v_1|\ldots|v_{k-1}|,\quad\forall v_1,\ldots,v_{k-1}\in\mathbb{R}^n.
\end{equation}

\subsection{The standard definition of the exterior derivative}
Let us briefly discuss the usual notion of exterior derivative in $\mathbb{R}^n$ from vector calculus.

\begin{dfn}[Standard exterior derivative]\label{dfn: dif_standard}
Let $\omega$ be a $(k-1)$-form on some open subset $U$ of $\mathbb{R}^n$. If $\omega$ is differentiable at some point $x\in U$, then we define the {\textbf{(usual) exterior derivative}} of $\omega$ at $x$ to be the $k$-linear form $\ed\omega_x$ in $\mathbb{R}^n$ given by
$$
\ed\omega_x(v_1,\ldots,v_k) := \sum_{i=1}^{k} (-1)^{i+1}\omega'(x)v_i(v_1,\ldots,\hat{v_i},\ldots,v_k),\quad\forall v_1,\ldots,v_k\in\mathbb{R}^n
$$
\end{dfn}

An equivalent formulation, always assuming $\omega$ is differentiable at $x\in U$, is to define the exterior derivative at $x$ in canonical
coordinates   
$\displaystyle
\omega = \sum_I \omega_I \ed x^I$, by
$$
\ed\omega_x(v_1,\ldots,v_k) := \sum_{I,j} \frac{\partial\omega_I}{\partial x^j}(x)\ed x^j\wedge {\ed}x^I(v_1,\ldots,v_k),\quad\forall v_1,\ldots,v_k\in\mathbb{R}^n.
$$
The equivalence
is straightforward, writing  $v_i = \sum_{j=1}^n v_i^j e_j$ in the canonical basis, for each ${i=1,\ldots,k}$, and noting that:
\begin{itemize}
\item for each $i=1,\ldots,k$, 
$$
\omega'(x)v_i =\sum_{j=1}^n v_i^j\omega'(x)e_j = \sum_{I,j}v_i^j\frac{\partial\omega_I}{\partial x^j}{\ed}x^I;
$$
\item for each multi-index $I=\{i_1<\ldots<i_{k-1}\}$ and $j=1,\ldots,n$, we have
\begin{eqnarray*}
{\ed}x^j\wedge {\ed}x^I(v_1,\ldots,v_k) &=& \begin{vmatrix}
v_1^j &  v_2^j  & \ldots & v_k^j\\
v_1^{i_1}  &  v_2^{i_1} & \ldots & v_k^{i_1}\\
\vdots & \vdots & \ddots & \vdots\\
v_1^{i_k}  &  v_2^{i_k} &\ldots & v_k^{i_k}
\end{vmatrix}\\
&=& \sum_{i=1}^k(-1)^{i+1}v_i^j {\ed}x^I(v_1,\ldots,\hat{v_i},\ldots,v_k).
\end{eqnarray*}
\end{itemize}

\subsection{The infinitesimal flux interpretation}

Before presenting Acker’s definition of the exterior derivative (cf.~\cite[p.~279]{acker1}), we first provide a heuristic argument, motivated by attempting to reverse-engineer the property we wish to obtain. The classical Stokes’ theorem (see \cite[Theorem 4-13, p. 102]{spivak}) for a differential $(k-1)$-form $\omega$ of class $C^1$ on an open set $U \subset \mathbb{R}^n$ asserts that, for every $C^1$ singular $k$-block $c:B^k \to U$, one has
\[
\int_{\partial c} \omega = \int_{c} \ed \omega,
\]
which can be rewritten as
\[
\int_{\partial B} c^{\ast}\omega = \int_{B} c^{\ast}(\ed \omega).
\]
Fixing a point $x \in U$ and considering the limit in which the singular $k$-block $c$ shrinks to $x$, so that $B$ shrinks to a point $p \in B$ with $c(p) = x$, then the continuity of $c^{\ast}(\ed\omega)=\ed(c^{\ast}\omega)$ yields
\[
\ed\omega_{x}(c'(p)e_1,\ldots,c'(p)e_k)
   = \lim_{B \to p} \frac{1}{\mathrm{vol}(B)}
     \int_{\partial B} c^{\ast}\omega.
\]
This suggests that the exterior derivative at $x$ may be interpreted as the ``flux density'' of $\omega$ through infinitesimal $k$-blocks around $x$. Motivated by this observation, and in the spirit of Lemma~\ref{lemma: 0-forms dim 1}, we introduce the following definitions.

\begin{dfn}\label{dfn: flux-integrable}
Let $U \subset \mathbb{R}^n$ be open and $x \in U$. A $(k-1)$-form $\omega$ on $U$ is said to be \textbf{flux-integrable around $x$} if, for every $C^1$ map $\varphi: A \to U$ from an open subset $A \subset \mathbb{R}^k$ with $\varphi(p) = x$, and every $k$-block $B \subset A$ with $p \in B$, the pullback $\varphi^{\ast}\omega$ is integrable over $\partial B$.
\end{dfn}

\begin{rem}[Flux-integrability of locally bounded forms with $(k-1)$-null discontinuities]
\label{rem:flux-null}
Let $U \subset \mathbb{R}^n$ be open and let $\omega \in \Omega^{k-1}(U)$ be \emph{locally bounded}, with $k>1$. 
Let $S \subset U$ denote the set of discontinuities of $\omega$, and assume that $S$ is \emph{$(k-1)$-null}: for every $\varepsilon>0$ there exists a countable collection of balls $\{B(x_i,r_i)\}_{i=1}^\infty$ in $\mathbb{R}^n$ such that
\[
S \subset \bigcup_{i=1}^\infty B(x_i,r_i), 
\qquad
\sum_{i=1}^\infty r_i^{\,k-1} < \varepsilon.
\]

\medskip
\noindent
We claim that $\omega$ is flux-integrable around every point of $U$, in the sense of Definition~\ref{dfn: flux-integrable}.

\medskip
\noindent
\textbf{Step 1: Images of null sets under Lipschitz maps.} 
If $E \subset \mathbb{R}^{k-1}$ is $(k-1)$-null and $f:E \to \mathbb{R}^n$ is Lipschitz with constant $L$, then $f(E)$ is $(k-1)$-null: covering $E$ by balls $\{B(x_i,r_i)\}$ with $\sum r_i^{\,k-1}<\varepsilon$, we have $f(E) \subset \bigcup_i B(f(x_i), L r_i)$, and $\sum_i (L r_i)^{\,k-1} \le L^{\,k-1} \varepsilon$.

\medskip
\noindent
\textbf{Step 2: Preimages under $C^1$ maps of maximal rank on compact domains.} 
Let $K \subset \mathbb{R}^{k-1}$ be compact and $f: K \to \mathbb{R}^n$ be $C^1$. Define
\[
R := \{x \in K : \operatorname{rank} Df(x) = k-1\}, \qquad S' := K \setminus R.
\]
At each $x \in R$, the Rank Theorem guarantees the existence of a neighbourhood $U_x \subset K$ such that $f|_{U_x}$ is a $C^1$ diffeomorphism onto its image, and in particular bi-Lipschitz. Covering the compact set $R$ with finitely many such neighbourhoods $\{U_{x_\alpha}\}_{\alpha=1}^m$, we deduce that for any $(k-1)$-null set $T \subset f(K)$, each preimage $f^{-1}(T)\cap U_{x_\alpha}$ is $(k-1)$-null; hence the union over $\alpha$ is $(k-1)$-null. The preimage over $S'$ is treated separately in Step~3.

\medskip
\noindent
\textbf{Step 3: Application to flux-integrability.} 
Fix $x \in U$ and let $\varphi: A \to U$ be $C^1$ with $\varphi(p)=x$, and let $B=\prod\limits_{i=1}^k[a_{i0},a_{i1}] \subset A$ be a $k$-block containing $p$. For each $i\in\{1,\dots,k\}$ and $j\in\{0,1\}$, let $F_{ij}$ be the smooth insertion map defining the faces of $B$ as in Definition~\ref{dfn:boundary}, and denote by 
\[
\Delta^{k-1} \coloneqq \prod_{l\neq i} [a_{l0},a_{l1}]
\] 
the $(k-1)$-block domain of the face. Define
\[
c_{ij} \coloneqq \varphi \circ F_{ij}: \Delta^{k-1} \to U.
\] 
Then each $c_{ij}$ is $C^1$ because it is the composition of the $C^1$ map $\varphi$ with the smooth map $F_{ij}$.  

Let $R_{ij} \subset \Delta^{k-1}$ denote the points where $Dc_{ij}$ has maximal rank $k-1$, and $S'_{ij} = \Delta^{k-1}\setminus R_{ij}$. Then:

\begin{itemize}
    \item On $R_{ij}$, $c_{ij}$ is a $C^1$ immersion. By Step~2, it is locally bi-Lipschitz on a finite covering of the compact face, so the preimage $c_{ij}^{-1}(S) \cap R_{ij}$ is $(k-1)$-null.
    \item On $S'_{ij}$, each $(k-1)$-minor of $Dc_{ij}$ vanishes. The coefficient functions of $c_{ij}^*\omega$ are finite linear combinations of these minors multiplied by the locally bounded coefficients of $\omega$, so all the coefficients vanish on $S'_{ij}$. Hence the pullback form $c_{ij}^*\omega$ itself vanishes identically on $S'_{ij}$ and is continuous there.
\end{itemize}

Hence, the discontinuities of $c_{ij}^*\omega$ lie in $(c_{ij}^{-1}(S) \cap R_{ij}) \cup S'_{ij}$. The first set is $(k-1)$-null, and the second contributes zero to the integral. By the classical Riemann integrability criterion, $c_{ij}^*\omega$ is integrable over $\Delta^{k-1}$. Summing over all faces, $\varphi^*\omega$ is integrable over $\partial B$.

\medskip
\noindent
\textbf{Conclusion.} 
Locally bounded $(k-1)$-forms whose discontinuity set is $(k-1)$-null are flux-integrable. Discontinuities contribute nothing to boundary integrals: either the pullback vanishes at points of non-maximal rank, or the preimage of discontinuities under the local $C^1$ immersions of maximal rank is $(k-1)$-null.
\end{rem}

\begin{dfn}[Exterior derivative as infinitesimal flux]\label{dfn: flux}
Let $\omega \in \Omega^{k-1}(U)$, where $U \subset \mathbb{R}^n$ is open. Suppose that $\omega$ is flux-integrable around $x\in U$, in the sense of Definition \ref{dfn: flux-integrable}. We say that $\omega$ is \textbf{derivable at $x$} if there exists a $k$-form $\eD\omega_x\in\Lambda^k(\mathbb{R}^n)^{\ast}$ such that, for every choice of $v_1,\ldots,v_k \in \mathbb{R}^n$,
\begin{equation}\label{dfn: D}
\eD\omega_x(v_1,\ldots,v_k)
  = \lim_{B \to p} \frac{1}{\mathrm{vol}(B)}
    \int_{\partial B} \varphi^{\ast}\omega,
\end{equation}
for every $C^1$ map $\varphi: A \to U$ from an open subset $A$ of $\mathbb{R}^k$ with $\varphi(p)=x$, $p\in A$, and $\varphi'(p)e_i = v_i$ for $i = 1,\ldots,k$. The limit in~\eqref{dfn: D} is understood in the following sense: for every $\varepsilon > 0$ and $K > 1$, there exists $\delta > 0$ such that, whenever $B \subset A$ is a $k$-block containing $p$ with $L(B) < \delta$ and $L(B)/l(B) < K$ (where $L(B)$ and $l(B)$ denote the lengths of the longest and shortest sides of $B$, respectively), one has
\[
\left|
  \eD\omega_x(v_1,\ldots,v_k)
  - \frac{1}{\mathrm{vol}(B)}\int_{\partial B} \varphi^{\ast}\omega
\right| < \varepsilon.
\]
In this case, $\eD\omega_x$ is called the \textbf{exterior derivative of $\omega$ at $x$}.
\end{dfn}

\textbf{Remarks:}\begin{itemize}
\item Although Definition~\ref{dfn: flux} requires the limit to be taken uniformly over all $C^1$ parametrisations $\varphi$ with prescribed first-order data at $p$, this condition is natural from the geometric viewpoint: it ensures that the resulting operator is compatible with arbitrary singular chains and depends only on the first-order behaviour of $\varphi$ at $p$.

\item The condition $L(B)/l(B)<K$ ensures a uniform bound on the ratio $\text{diam}(B)\cdot\text{area}(\partial B)/\text{vol}(B)$; this will play an important role in showing that the above definition of the exterior derivative coincides with the usual one in the differentiable case (Theorem \ref{thm: compatibility}).

\item This definition is very geometric: the exterior derivative of $\omega$ at a point $x$ is the limit of the ``average'' flux of $\omega$ across the boundary of small blocks shrinking (in a controlled way) to $x$. This contrasts with Definition \ref{dfn: dif_standard}, which involves a somewhat computational notion of differentiability. Here lies an important point: the distinction between the notions of differentiability and derivability.

\item The notion of derivability is strictly weaker than classical differentiability. In particular, $\omega$ need not be continuous at $x$ for $\eD\omega_x$ to exist. For instance, the $1$-form $\omega$ defined on $\mathbb{R}^2$ by $\omega=f\ed x^1$, where
$$
f(x_1,x_2)=\left\{\begin{array}{rc}
1, &\mbox{if}\quad (x_1,x_2)=(0,0),\\
0, &\mbox{if}\quad (x_1,x_2)\neq(0,0),
\end{array}\right.
$$
is discontinuous at the origin but has a well-defined exterior derivative $\eD\omega_{(x_1,x_2)}=0$ at every point $x=(x_1,x_2)\in\mathbb{R}^2$; see Example \ref{ex:constant-null}. 
\end{itemize}

\begin{example}[Locally bounded, almost constant $(k-1)$-forms with $(k-1)$-null discontinuities]
\label{ex:constant-null}

Let $k>1$ and $U \subset \mathbb{R}^n$ be an open set. Let $S \subset U$ be a closed $(k-1)$-null set, and consider a locally bounded $(k-1)$-form
\[
\omega = \sum_I \omega_I \ed x^I \in \Omega^{k-1}(U),
\]
whose coefficient functions $\omega_I$ are \emph{constant} on $U \setminus S$, for every $I=\{i_1<\ldots<i_k\}\subset\{1,\ldots,n\}$.

Fix any point $x \in U$ and a $C^1$ map $\varphi: A \to U$ with $\varphi(p) = x$, and let $B \subset A$ be a $k$-block containing $p$. Denote by $c_{ij}$ the singular $(k-1)$-faces of $B$ as in Definition~\ref{dfn:boundary}. Then, for each face $c_{ij}$, the pullback $c_{ij}^*\omega$ is constant on the face, except possibly on the preimage $c_{ij}^{-1}(S)$. Now, under our assumptions, by the argument in Step 3 of Remark \ref{rem:flux-null}, the integral of the pullback over $c_{ij}^{-1}(S)$ vanishes. Moreover, as for the constant part of $c_{ij}^*\omega$, it gives equal and opposite contributions on opposite faces due to the alternating signs in the boundary sum:
\[
\int_{\partial B} \varphi^*\omega = \sum_{i=1}^k \sum_{j=0}^1 (-1)^{i+j} \int_{c_{ij}} c_{ij}^*\omega = 0.
\]

Hence, for any $B$,
\[
\int_{\partial B} \varphi^*\omega = 0,
\]
and taking the limit as $B \to p$ in Definition~\ref{dfn: flux} yields
\[
\eD \omega_x = 0.
\]
In conclusion, such locally bounded, almost constant $(k-1)$-forms with $(k-1)$-null discontinuities are derivable everywhere and have zero exterior derivative.
\end{example}

\subsection{The mean value equality for differential forms}
The results discussed in the remainder of Section \ref{sec: Exterior derivative} stem from the bibliography according
to a pattern: they are stated and argued for in \cite{acker1,paper2}  and proved
as a low-dimensional instance in \cite{paper}. We therefore propose a systematic exposition of the theory with self-contained and fully rigorous proofs of the general statements. Let us begin with the mean value theorem for differential forms: 

\begin{thm}[M.V.T. for differential forms]\label{thm: mvt}
Let $\omega$ be a $(k-1)$-form defined on a $k$-block $B=[a_1,b_1]\times\ldots\times[a_k,b_k]\subset\mathbb{R}^k$ such that:
\begin{enumerate}
\item $\omega$ is flux-continuous (Definition \ref{dfn: cont});
\item $\omega$ is derivable on $\mathring{B}$ (Definition \ref{dfn: flux});
\end{enumerate} 
then there exists $\xi\in\mathring{B}$ such that 
$$
{\eD\omega}_\xi(e_1,\ldots,e_k) = \frac{1}{\mathrm{vol}(B)}\int_{\partial{B}}\omega.
$$
\end{thm}

\begin{proof}
Applying the Trisection Lemma \ref{thm: fundlemma}, we obtain a descending chain of non-empty compact sets $(B_n)_{n\in\mathbb{N}}$ such that, for each $n\in\mathbb{N}$, we have:
\begin{enumerate}
\item[(i)] $B_n \subset \mathring{B}$;
\item[(ii)] the sides of $B_{n+1}$ measure $1/3$ of the sides of $B_n$;
\item[(iii)] 
$\displaystyle
\frac{1}{\text{vol}(B_n)}\int_{\partial{B_n}}\omega
= \frac{1}{\text{vol}(B)}\int_{\partial{B}}\omega.$
\end{enumerate}

In particular, $\{B_n\}$ is a family of closed sets in the compact subspace $B$, satisfying the finite intersection property, so by the well-known characterization of compact sets we have $\bigcap{B_n}\neq\emptyset$. Moreover, as $\text{diam}(B_n)\rightarrow 0$, there exists a unique element $\xi\in\bigcap{B_n}$, i.e. $\{\xi\}=\bigcap{B_n}$. Since, by (i),  $\xi\in\mathring{B}$  and $\omega$ is derivable in the interior of $B$, we claim that conditions (ii) and (iii) on $(B_n)$  give
$$ 
\eD\omega_{\xi}(e_1,\ldots,e_k) = \lim_{n\to\infty} \frac{1}{\text{vol}(B_n)}\int_{\partial{B_n}}\omega = \frac{1}{\text{vol}(B)}\int_{\partial{B}}\omega, 
$$ in the terms of Definition \ref{dfn: flux}. Indeed, by (iii) we only
need to prove the first equality. Fix 
$$
K>\frac{\max\{|b_i-a_i|:i=1,\ldots,k\}}{\min\{|b_i-a_i|:i=1,\ldots,k\}}\geq 1
$$ and let $\varepsilon>0$. By the derivability of $\omega$ on $\mathring{B}\ni\xi$, (taking $\varphi=1_{\mathbb{R}^k}$) there exists $\delta>0$ such that, if $C$ is a $k$-block inside $\mathring{B}$ with $L(C)<\delta$, $\xi\in C$, and $L(C)/l(C)<K$, we have
$$
\left|\eD\omega_{\xi}(e_1,\ldots,e_k) - \frac{1}{\text{vol}(C)}\int_{\partial C}\omega\right|<\varepsilon.                 
$$

Now, by  (ii), there exists $n_0\in\mathbb{N}$ such that $L(B_{n})<\delta$ for every $n\geq n_0$. Moreover, by definition of $K$ and $\xi$, we have $L(B_n)/l(B_n)<K$ and $\xi\in B_n$, for every $n\in\mathbb{N}$. Therefore,
$$
n\geq n_0\Rightarrow 
\left|\eD\omega_{\xi}(e_1,\ldots,e_k) - \frac{1}{\text{vol}({B_n})}\int_{\partial{B_n}}\omega\right|<\varepsilon,
$$
which proves the claim.
\end{proof}

With the same technique as in the one-dimensional Theorem \ref{thm: one_dim}, we have proved a generalised mean value theorem in equality form. Thus, we also obtain the following higher-dimensional version of Theorem \ref{thm: ftc}:

\begin{prop}[Divergence theorem]\label{prop: stokes}
If $\omega$ is a $(k-1)$-form defined on a $k$-block $B\subset\mathbb{R}^k$ such that:
\begin{enumerate}
\item $\omega$ is flux-continuous;
\item $\omega$ is derivable on $\mathring{B}$;
\item $\eD\omega$ is (Riemann) integrable on $B$,
\end{enumerate}
then $$ \int_{\partial B}\omega = \int_B \eD\omega. $$
\end{prop}
\begin{proof}
Indeed, we start by taking a partition $P$ of $B$. Let $S(P)$ be the collection of sub-blocks determined by the partition $P$ of $B$. Thus, 
$$ \int_{\partial B} \omega = \sum\limits_{\beta\in S(P)} \int_{\partial\beta} \omega. $$ Now, by the previous version of the MVT  (Theorem \ref{thm: mvt}), in the interior of each $\beta\in S(P)$ there exists $\xi_{\beta}$ such that 
$$
\int_{\partial\beta} \omega = \eD\omega_{\xi_{\beta}}(e_1,\ldots,e_k)\text{vol}(\beta). 
$$ Therefore, we have 
$$ 
s(\eD\omega(\cdot{})(e_1,\ldots,e_k);P)\leq \int_{\partial B} \omega \leq S(\eD\omega(\cdot{})(e_1,\ldots,e_k);P). 
$$ The result follows by the generality of $P$ and the (Riemann) integrability of $D\omega$ on $B$.
\end{proof}

One can cast these last results in a more general setting, given the good behavior of $\eD$ under pullbacks (claimed without
proof in \cite[p.~286]{acker1}): 

\begin{lemma}[Naturality of $\eD$]\label{thm: pullback_lemma}
Let $\psi:V\rightarrow U$ be any map of class $C^1$ between open subsets $V\subseteq\mathbb{R}^p$ and $U\subseteq\mathbb{R}^n$, and let $\omega\in\Omega^{k-1}(U)$ be a $(k-1)$-form on $U$. If $\omega$ is derivable at $x=\psi(t)$, $t\in V$, then $\psi^{\ast}\omega$ is derivable at the point $t$ and $\eD(\psi^{\ast}\omega)(t)=\psi^{\ast}(\eD\omega)(t)$.
\end{lemma}

\begin{proof}
We want to prove that the  $(k-1)$-form  $\psi^{\ast}\omega$ defined on  $V\subset\mathbb{R}^p$ is derivable at  $t$. So let $\varphi: A\rightarrow V$ be any $C^1$ map, with $A$ an open subset of $\mathbb{R}^k$, such that $\varphi(s)=t$ for some $s\in A$. Then, $\psi\circ\varphi:A\rightarrow U$ is a $C^1$ map taking $s$ to $x$. Moreover, since $\omega$ is derivable at $x=\psi(t)$, we can compute $\eD\omega_{\psi(t)}((\psi\circ\varphi)'(s)e_1,\ldots,(\psi\circ\varphi)'(s)e_k)\in\mathbb{R}$. The standard property of the pull-back under composition gives
$(\psi\circ\varphi)^{\ast}\omega =\varphi^{\ast}(\psi^{\ast}\omega) $, hence

\begin{eqnarray*}
\eD\omega_{\psi(t)}((\psi\circ\varphi)'(s)e_1,\ldots,(\psi\circ\varphi)'(s)e_k)
&=& \lim_{\tilde{B}\to s} \frac{1}{\text{vol}(\tilde{B})}\int_{\partial\tilde{B}}(\psi\circ\varphi)^{\ast}\omega\\
&=& \lim_{\tilde{B}\to s} \frac{1}{\text{vol}(\tilde{B})}\int_{\partial\tilde{B}}\varphi^{\ast}(\psi^{\ast}\omega)
\end{eqnarray*} and the conclusion  follows straight from Definition \ref{dfn: flux}.
\end{proof}

The naturality of $\eD$ makes rigorous the following generalised version of Stokes' theorem, claimed in \cite[p.~281, Theorem A]{acker1} (see also \cite[p.7]{paper2} for a manifold version):
\begin{thm}[Stokes for $\eD$]
\label{thm: gstokes}
Let $U$ be an open set of $\mathbb{R}^n$, $\omega$ a $k$-form defined on $U$ and $c:B\subset\mathbb{R}^k\rightarrow U$ a $k$-singular block of class $C^1$ and suppose the following conditions hold:
\begin{enumerate}
\item[(1)] $c^{\ast}\omega$ is flux-continuous on $B$; 
\item[(2)] $\omega$ is derivable on $c(\mathring{B})$;
\item[(3)] $c^{\ast}(\eD\omega)$ is (Riemann) integrable on $B$.
\end{enumerate}
Then, $$ \int_{\partial c}\omega = \int_{c}\eD\omega.$$
\end{thm}

\begin{proof}
From (2) and the fact $c$ is of class $C^1$, we can apply Lemma \ref{thm: pullback_lemma} to $\psi=c|_{\mathring{B}}$ and conclude that $c^{\ast}\omega$ is derivable and $\eD(c^{\ast}\omega)=c^{\ast}(\eD\omega)$ on $\mathring{B}$. In particular, in view of (3), it follows that $\eD(c^{\ast}\omega)$ is integrable on $B$. Moreover, by (1), $c^{\ast}\omega$ is flux-continuous on $B$, so we are in the conditions of Theorem \ref{prop: stokes}: $$ \int_{\partial B} c^{\ast}\omega = \int_{B} \eD(c^{\ast}\omega). $$ On the other hand, as $\partial B$ has measure zero in $\mathbb{R}^k$ and $B=\mathring{B}\cup\partial B$, we have $$\int_{B} \eD(c^{\ast}\omega) =  \int_{B}c^{\ast}(\eD\omega) = \int_{c} \eD\omega, $$ and (again) by the fact that $c$ is of class $C^1$, 
$$ \int_{\partial B} c^{\ast}\omega = \int_{\partial c} \omega $$ which concludes the proof.    
\end{proof}




\begin{thm}[$\eD^2=0$]
\label{thm: D^2=0}
Let $\omega$ be a $(k-1)$-form continuous on some open subset $U$ of $\R^n$. Suppose also $\omega$ is derivable on $U$ and that $\eD\omega$ is integrable in every $k$-chain of class $C^1$ on $U$. Then, $\eD\omega$ is derivable on $U$ and 
$$
\eD^2\omega=0.
$$
\end{thm}
\begin{proof}
Let $c$ be any $(k+1)$-chain of class $C^1$ on $U$. Then, the hypotheses ensure we can apply Theorem \ref{thm: gstokes} to $\omega$ using the $k$-chain $\partial c$:
$$
\int_{\partial c}\eD\omega=\int_{\partial(\partial c)}\omega =0.
$$
since  $\partial^2=0$ (cf. p.\pageref{page
boundary2=0}). In the terms of Definition
\ref{dfn: flux}, generality of $c$ implies $\eD(\eD\omega)_x=0$ at any given
$x\in U$.                         \end{proof}
In particular, Theorem \ref{thm: D^2=0}
points to a natural cohomology theory associated to $\eD$; it is indeed an
interesting question whether this extension contains any information beyond standard de Rham cohomology, since a priori one has
additional closed forms (which could also be compensated by the new exact forms as well). Here one has in mind not only open subsets in $\mathbb{R}^n$ but also the natural extension of this theory to smooth $n$-manifolds. Given the introductory scope of this article, we
leave it as a meditation topic for the motivated reader.                
\subsection{Consistency with the Classical Exterior Derivative}
\label{subsec: compatibility}

Finally, we conclude that $\eD$ indeed extends $\ed$ in the usual differentiable
cases. The rest of this Subsection is devoted to the proof of  the following
result:                 \begin{thm}[$\eD=\ed$ in the differentiable case]
\label{thm: compatibility}
If $\omega$ is a $(k-1)$-form on $\mathbb{R}^n$ defined on some open subset $U$ such that $\omega$ is differentiable at a point $x_0\in U$, then $\omega$ is derivable at this point (in the sense of Definition \ref{dfn: flux}) and $$\eD\omega_{x_0} = \ed\omega_{x_0}.$$ 
\end{thm}

\textbf{Proof.} By the differentiability of $\omega$ at the point $x_0$, we may write
\begin{equation} \label{eq: eq1}
\omega(x) = \omega(x_0) + \omega'(x_0)(x-x_0) + \rho(x)|x-x_0|,
\end{equation} 
where $\displaystyle\lim_{x\rightarrow x_0}\rho(x)= 0$.
Define $\omega_0,\omega_1\in\Omega^{k-1}$ by 
$$
\omega_0(x)\coloneqq \omega(x_0) + \omega'(x_0)(x-x_0), \quad \omega_1(x)\coloneqq \rho(x)|x-x_0|.
$$
Certainly  $\omega_0$ is of class $C^1$, and $\omega_0'(x) = \omega'(x_0)$, $\forall x\in U$. Indeed, by the linearity of $\omega'(x_0)$, 
\begin{eqnarray*}
\left\Vert\omega_0(x+h)-\omega_0(x)-\omega'(x_0)h\right\Vert &=& \left\Vert\omega(x_0)+\omega'(x_0)(x+h-x_0) - \omega(x_0) - \omega'(x_0)(x-x_0) - \omega'(x_0)h\right\Vert\\
&=& \left\Vert(\omega'(x_0)(x-x_0) + \omega'(x_0)h) - \omega'(x_0)(x-x_0) - \omega'(x_0)h\right\Vert\\
&=& 0.\\
\end{eqnarray*}

Now, let $v_1,\ldots,v_k\in\mathbb{R}^n$ and $\varphi:A\to U$ be a $C^1$ with $\varphi(t_0)=x_0$ and $\varphi'(t_0)e_i=v_i$, for each $i=1,\ldots,k$. We want to prove that 
\begin{equation} \label{eq: eq_provar}
\lim_{B\to t_0} \frac{1}{\text{vol}(B)}\int_{\partial B}\varphi^{\ast}\omega = \sum_{i=1}^k (-1)^{i+1}\omega'(x_0)v_i(v_1,\ldots,\widehat{v_i},\ldots,v_k). 
\end{equation}

By (\ref{eq: eq1}), $\omega = \omega_0 + \omega_1$. On the other hand, using the classical Stokes theorem for $\omega_0$ and
the naturality of $\ed$: 
\begin{eqnarray*}
\lim_{B\to t_0} \frac{1}{\text{vol}(B)}\int_{\partial B}\varphi^{\ast}\omega_0 &=& \lim_{B\to t_0} \frac{1}{\text{vol}(B)}\int_{B}\varphi^{\ast}(\ed\omega_0)  \\
&=& \varphi^{\ast}(\ed\omega_0)|_{t_0} \\
&=& \sum_{i=1}^k (-1)^{i+1}\omega'(x_0)v_i(v_1,\ldots,\widehat{v_i},\ldots,v_k), 
\end{eqnarray*} since $\omega_0'(x)=\omega'(x_0)$. The last equation is exactly  right-hand side of (\ref{eq: eq_provar}). So we are left to prove that 
$$
\lim_{B\to t_0} \frac{1}{\text{vol}(B)}\int_{\partial B}\varphi^{\ast}\omega_1 = 0,
$$ i.e. that given $\varepsilon>0$ and $K>1$, there exists $\delta>0$ such that, for every $k$-block $B$ in $A$ satisfying
\begin{equation} \label{eq: block}
L(B)<\delta, \quad t_0\in B, \quad\frac{L(B)}{l(B)}<K,
\end{equation} one has
$$
\left|\frac{1}{\text{vol}(B)}\int_{\partial B}\varphi^{\ast}\omega_1\right|<\varepsilon.
$$

Before we proceed, observe that if $B$ is a $k$-block satisfying the condition $L(B)/l(B)<K$, then by the fact that $\text{vol}(B)\geq l(B)^k$ (recall $l(B)$ is the smaller side of $B$), we obtain 
\begin{equation} \label{eq: eq_bound1}
\frac{1}{\text{vol}(B)}<\frac{K^k}{L(B)^k}.
\end{equation}
The integral we want to control (for sufficiently small blocks $B$) has an integrand of the form 
$$
(\varphi^{\ast}\omega_1)_{t}(v_1,\ldots,v_{k-1}) = \rho(\varphi(t))(\varphi'(t)v_1,\ldots,\varphi'(t)v_{k-1})|\varphi(t)-\varphi(t_0)|,
$$ so we are lead to bound the norm of $\varphi'$, in some neighbourhood of $t_0$, by some constant $M>0$. Once we do that, we can use inequality (\ref{eq: des1}) and the classical mean value inequality to bound the last term by $\left\Vert\rho(\varphi(t))\right\Vert M^{k-1}|v_1|\ldots|v_{k-1}|M|t-t_0|$ (for $t$ sufficiently near $t_0$). Thus, with \textit{normalised} vectors $v_1,\ldots,v_{k-1}$, and using the elementary fact\footnote{Just think in the case of a $k$-cube with side $L(B)$} that $\text{diam}(B)\leq L(B)\sqrt{k}$, we will conclude that, for $t$ \textit{sufficiently near} $t_0$,
\begin{equation}\label{eq: eq_bound2}
|\rho(\varphi(t))(\varphi'(t)v_1,\ldots,\varphi'(t)v_{k-1})||\varphi(t)-\varphi(t_0)|\leq \left\Vert\rho(\varphi(t))\right\Vert M^kL(B)\sqrt{k}.
\end{equation}

Let $\varepsilon>0$ and $K>1$ be given. We want to construct the neighbourhood of $t_0$ where not only does the above equation hold, but also the term $\left\Vert\rho(\varphi(t))\right\Vert$ is bounded by a suitable constant. First, as $\varphi'(t_0)$ is a linear transformation between finite-dimensional vector spaces, there exists $M>0$ such that $\left\Vert\varphi'(t_0)\right\Vert=\text{sup}\{|\varphi'(t_0)v|:|v|=1\}<M$. Now, $\varphi$ is of class $C^1$ (i.e. $\varphi'$ is continuous), so we can find an open neighbourhood $A_0$ of $t_0$ such that $\left\Vert\varphi'|_{A_0}\right\Vert<M$,
for example  
$$
A_0\coloneqq (\left\Vert\cdot{}\right\Vert\circ\varphi')^{-1}(]-1,M[).
$$ 
On the other hand, since $\rho(x)$ gets small when $x\rightarrow x_0$, there exists a neighbourhood $V$ of $x_0 $ in  $ U$ such that
$$
x\in V \Rightarrow \left\Vert\rho(x)\right\Vert<\frac{\varepsilon}{2k\sqrt{k}M^kK^k}
$$ Thus, by continuity of $\varphi$ and the fact that $\rho(x_0) = 0$, $A_1\coloneqq\varphi^{-1}(V)$ is an open neighbourhood of $t_0$ in $A$. 

Now choose $\delta_0>0$ such that the Euclidean ball of center $t_0$ and radius $\delta_0$ satisfies  $B(t_0;\delta_0)\subseteq A_0\cap A_1$ and set $\delta\coloneqq\delta_0/{\sqrt{k}}$. By construction, if $B$ is a $k$-block in $A$ such that $t_0\in B$ and $L(B)<\delta$, then $B\subseteq A_0\cap A_1$, and
\begin{equation} \label{eq: eq_bound3}
t\in B\Rightarrow \left\Vert\varphi'(t)\right\Vert<M,\quad \left\Vert\rho(\varphi(t))\right\Vert<\frac{\varepsilon}{2k\sqrt{k}M^kK^k}.
\end{equation}
Writing $B_j$ for the faces of $\partial B$, and $\{v_{j1},\ldots,v_{jk-1}\}$ for normalised generators of $B_j$, ${j=1,\ldots,2k}$, the inequalities (\ref{eq: eq_bound2}) and (\ref{eq: eq_bound3}) together imply
\begin{equation} \label{eq: eq_bound4}
|\rho(\varphi(t))(\varphi'(t)v_{j1},\ldots,\varphi'(t)v_{jk-1})||\varphi(t)-\varphi(t_0)|\leq\frac{\varepsilon.L(B)}{2kK^k},\quad\forall t\in B.
\end{equation}

Therefore, if $B$ satisfies (\ref{eq: block}), we have

\begin{eqnarray*}
\left|\int_{B_j}\varphi^{\ast}\omega_1\right| &=& \left|\int_{B_j}\rho(\varphi(t))(\varphi'(t)v_{j1},\ldots,\varphi'(t)v_{jk-1})|\varphi(t)-\varphi(t_0)|dt\right| \\
&\leq& \frac{\varepsilon.L(B)}{2kK^k}.L(B)^{k-1},
\end{eqnarray*} using (\ref{eq: eq_bound4}) and the fact that $\text{vol}(B_j)\leq L(B)^{k-1}$. Combining this estimate
with the bound (\ref{eq: eq_bound1}) on $\vol(B)$, we obtain

\begin{eqnarray*}
\left|\frac{1}{\text{vol}(B)}\int_{\partial B}\varphi^{\ast}\omega_1\right| 
&<& 
\frac{K^k}{L(B)^k}\sum_{j=1}^{2k}\left|\int_{B_j}\varphi^{\ast}\omega_1\right|\\
&<& \frac{K^k}{L(B)^k}.(2k)\frac{\varepsilon.L(B)}{2kK^k}.L(B)^{k-1}\\
&=& \varepsilon
\end{eqnarray*}
provided $k>1$ and so $|\int_{B_j}|\leq\int_{B_j}|\cdot{}|$.

The  $k=1$ case is just Lemma \ref{lemma: 0-forms dim 1}, presented in the
Introduction, so for clarity, let us write the proof in its original notation.
We prove ($\Rightarrow$), as the implication ($\Leftarrow$) is clear. By the differentiability of $f$ at $x$ we may write, for every $y\in U$,
\begin{equation}
\label{eq: first order}
f(y)=f(x)+f'(x)(y-x)+|y-x|\rho(y),\quad\text{with}\quad\lim_{y\to x}\rho(y)=0.
\end{equation}
So  $f=f_0+f_1$, with $f_0(y)\coloneqq f(x)+f'(x)(y-x)$ and $f_1(y)\coloneqq |y-x|\rho(y)$, for each $y\in U$. Now, observe that $f_0$ is infinitely differentiable on $U$ with constant derivative $f_0'(y)\equiv f'(x)$ for each $y\in U$. Thus, for every $[a,b]\subseteq U$, we can apply the second fundamental theorem of calculus to $f$:
$$
\frac{1}{b-a}\int_{\partial[a,b]}f_0 = \frac{1}{b-a}\int_{[a,b]}f_0' = f'(x).
$$ 
We are left to show that 
$$
\lim_{[a,b]\to x}\frac{1}{\text{vol}[a,b]}\int_{\partial [a,b]}f_1 = 0.
$$
Let $\varepsilon>0$ be given. By (\ref{eq: first order}), there exists some $\delta>0$ such that $|\rho(y)|<\varepsilon/2$ for every $y\in U$ such that $|y-x|<\delta$. Then, if $[a,b]\subseteq U$ is such that $x\in[a,b]$ and $b-a<\delta$, we can write
\begin{eqnarray*}
\left|\frac{1}{b-a}(f_1(b)-f_1(a))\right| &=& \frac{1}{b-a}\big\lvert{|b-x|\rho(b) - |a-x|\rho(a)}\big\rvert\\
&\leq& \frac{1}{b-a}\big((b-a)\lvert\rho(b)\rvert+(b-a)\lvert\rho(a)\rvert\big)\\
&=& \lvert\rho(b)\rvert+\lvert\rho(a)\rvert\\
&<& \varepsilon.
\end{eqnarray*} 
\section{Stokes' theorem}
\label{sec:Stokes_Theorem}
Perhaps the most important corollary of the last compatibility result (Theorem \ref{thm: compatibility}) with the $\eD$ version of Stokes' theorem (Theorem \ref{thm: gstokes}) is the following version of Stokes' theorem for $\ed$, which does not require $C^1$ regularity on the integrand:

\newtheorem*{thm: Stokes without C1}
{Theorem \ref{thm: Stokes without C1}}

\begin{thm: Stokes without C1}[Differentiable version of Stokes' theorem]\label{thm: diff_stokes}
Let $\omega\in\Omega^k(U)$ be a $k$-form on some open subset $U$ of $\mathbb{R}^n$, and let $c:B\rightarrow U$ be a $(k+1)$-singular block of class $C^1$ in $U$. If $\omega$ is continuous on $c(B)$ and differentiable on $c(\mathring{B})$, then
$$ \int_{\partial c} \omega = \int_{c} \ed\omega, $$ when the right-hand side integral exists.
More generally, let $c=n_1c_1+\ldots+n_lc_l$ is a $(k+1)$-chain of class $C^1$, with $c_i:B_i\rightarrow U$ ($i=1,\ldots,l$). If $\omega$ is continuous on $c(B_i)$ and differentiable on $c(\mathring{B_i})$, for each $i=1,\ldots,l$, then the same conclusion holds.
\end{thm: Stokes without C1}

We can easily derive some other well-known results. In what follows, by a (closed) rectangle in $\mathbb{R}^2$ we mean a possibly rotated and translated $2$-block, i.e. the image of a $2$-singular block $c:B\to\mathbb{R}^2$ of the form $c=E\circ i$, where $i:B\hookrightarrow\mathbb{R}^2$ is the inclusion of the $2$-block $B$ and $E:\mathbb{R}^2\to\mathbb{R}^2$ is a rigid motion in $\mathbb{R}^2$ (i.e. the composition of a translation with a rotation on the plane). In particular, we obtain (see \cite[p.~215]{paper}) the following version of Green's theorem (on rectangles):

\begin{thm}[Green's Theorem without $C^1$ assumption]\label{thm: green}
Let $R$ be a (closed) rectangle in $\mathbb{R}^2$ and $P,Q:R\rightarrow\mathbb{R}$ continuous functions on $R$ that are differentiable in the interior of $R$. Then,
$$ \int_{\partial R}{P\ed x + Q\ed y} = \int_R \left(\frac{\partial Q}{\partial x}-\frac{\partial P}{\partial y}\right)\ed x\ed y,$$ when the right-hand side integral exists. 
\end{thm}

\section{Numerical Exterior Derivative}
\label{sec:numerical-exterior-derivative}
Building on the exposition in Section~\ref{sec: Exterior derivative}, we propose a numerical method for computing the components of the exterior derivative $\eD \omega$ of a given $(k-1)$-form $\omega$ in an open set $U\subset\mathbb{R}^n$. Our method operates at any point $x \in U \subset \mathbb{R}^n$ using a black-box numerical sampler that does not require access to the analytical expressions of the components of $\omega$. In addition, a demo implementation in SageMath~\cite{sagemath} for exterior differentiation in $\mathbb{R}^3$ is available at \href{https://github.com/TomasSilva/NumericalExteriorDerivative}{github.com/TomasSilva/NumericalExteriorDerivative}.

Contrary to existing approaches, such as Discrete Exterior Calculus (DEC)~\cite{Hirani}, which computes the exterior derivative via discretization over a mesh, or Finite Element Exterior Calculus (FEEC)~\cite{FEEC}, that uses the structure of the discrete de Rham complex to approximate differentiation, our method treats $\omega$ as a black-box function and does not require any mesh decomposition. Indeed, we assume access only to numerical evaluations of $\omega$ as a skew-symmetric $(0, k-1)$-tensor represented by an $n \times \cdots \times n$ array with $k-1$ skew-symmetric slots at any given point in $\R^n$. 

Using the notation established in \eqref{eq:general_form}, for any increasing multi-index $I=\{i_1<\ldots<i_{k-1}\}\subseteq\{1,\ldots,n\}$ of lenght $k-1$, we recall that the $I$-th component function $\omega_I$ of $\omega \in \Omega^{k-1}(\mathbb{R}^n)$ is
\[
x\mapsto\omega_I(x) = \omega_x(\partial_{{i_1}}, \dots, \partial_{i_{k-1}}) \in \mathbb{R}.
\] 
For a given point $x \in \mathbb{R}^n$ and degree $k-1$, the black-box sampler described in Algorithm~\ref{alg:sampler} outputs an array containing the numerical values of $\omega_I(x)$ for all increasing multi-indices $I$, without requiring a symbolic expression for  $\omega$ at all.
\begin{algorithm}[]
\caption{\texttt{sampler}}\label{alg:sampler}
\begin{algorithmic}[1]
\Require A point \( x \in \mathbb{R}^n \) and a degree \( (k-1) \in \{0, 1, \dots, n\} \)
\Ensure A skew-symmetric \(k\)-tensor representing the numerical components of some \( \omega \in \Omega^{k-1}(\mathbb{R}^n) \)

\State Initialise an empty dictionary \( T \)
\ForAll{multi-indices \( (i_1, \dots, i_{k-1}) \) with \( 1 \le i_1 < \dots < i_{k-1} \le n \)}
    \State Sample the value \( a_{i_1 \dots i_{k-1}} = \omega_{i_1 \dots i_{k-1}}(x) \)\Comment{Sampling treated as black-box}
    \State Set \( T[i_1, \dots, i_{k-1}] \gets a_{i_1 \dots i_{k-1}} \)
\EndFor
\State \Return Anti-symmetrisation of \( T \)
\end{algorithmic}
\end{algorithm}
The components of the exterior derivative $\eD\omega_x \in \Lambda^{k}(\R^n)^{\ast}$ can then be numerically computed using Definition~\ref{dfn: flux}.
Let $Q=\{q_1<\ldots<q_k\}\subseteq\{1,\ldots,n\}$ be an increasing multi-index of length $k$. Taking the associated canonical vectors $(e_{q_1},\ldots,e_{q_k})$ in $\R^n$, and letting $\varepsilon>0$ be sufficiently small, we can construct a singular $k$-block $c:B\to\mathbb{R}^n$ whose image $c(B)$ is the small $k$-cube centered at $x$ with side length $2\varepsilon$ in the $k$-plane generated by $e_{q_1},\ldots,e_{q_k}$,  
and then use the approximation
\begin{equation}\label{eq: eD_approx}
\eD\omega_x(e_{q_1},\ldots,e_{q_k}) \approx \frac{1}{(2\varepsilon)^k}\int_{\partial B}c^{\ast}\omega.
\end{equation}
Explicitly, we take $B:=[-\varepsilon,\varepsilon]^k\subset\mathbb{R}^k$ and let $c:B\to\mathbb{R}^n$ be given by
\[
c(t_1,\ldots,t_k):=x+\sum_{i=1}^k t_ie_{q_i}.
\] By Definition~\ref{dfn:boundary}, for each $i\in\{1,\ldots,k\}$ and $j\in\{0,1\}$, defining 
\[
c_{ij}:[-\varepsilon,\varepsilon]_1\times\ldots\times\widehat{[-\varepsilon,\varepsilon]_i}\times\ldots\times[-\varepsilon,\varepsilon]_k\longrightarrow\mathbb{R}^n
\] by 
\[
c_{ij}(t_1,\ldots,\hat{t_i},\ldots,t_k)\coloneqq c(t_1,\ldots,(-1)^{j+1}\varepsilon,\ldots,t_k),
\] then the boundary of $c$ is the $(k-1)$-chain $\partial c$ is given by
$$ 
\partial c = \sum\limits_{i=1}^k\sum\limits_{j=0}^1(-1)^{i+j}c_{ij},
$$ where the images of the $c_{ij}$ are the $(k-1)$-dimensional faces of the $k$-cube $c(B)$, with $c_{i0}$ and $c_{i1}$ being opposite faces, and they have centers $x_{ij}:=x + (-1)^{j+1} \varepsilon e_{q_i}$. 


In order to approximate the integral of $c^{\ast}\omega$ over each face $B_{ij} \subset \partial B$, we evaluate $\omega$ at the face center 
$$
x_{ij}=c(0,\ldots,0,\underbrace{(-1)^{j+1}\varepsilon}_{\text{$i$-th position}},0,\ldots,0) =: c(t_{ij}),
$$ in the $(k-1)$-plane $(e_{q_1},\ldots,e_{q_k})$ and multiply by the $(k-1)$-dimensional volume:
\begin{align*}
\int_{B_{ij}} c^{\ast}\omega &\approx \mathrm{vol}(B_{ij}) \, \omega_{c(t_{ij})}(c'(t_{ij})e_1,\ldots, \widehat{c'(t_{ij})e_i} ,\ldots, c'(t_{ij})e_{k})\\
&= (2\varepsilon)^{k-1} \, \omega_{x + (-1)^{j+1}\varepsilon e_{q_i}}(e_{q_1},\ldots, \widehat{e_{q_i}} ,\ldots,e_{q_k}).
\end{align*}

Therefore, using the above in \eqref{eq: eD_approx}, for $\varepsilon > 0$ sufficiently small, each component of $\eD \omega_x$ can be numerically approximated by
\begin{align}
\label{eq:numerical_approx}
\eD\omega_x(e_{q_1},\ldots,e_{q_k}) =\frac{1}{(2\varepsilon)}\sum\limits_{i=1}^k\sum\limits_{j=0}^1(-1)^{i+j}\omega_{x + (-1)^{j+1} \varepsilon e_{q_i}}(e_{q_1},\dots, \widehat{e_{q_i}} ,\dots,e_{q_k}).
\end{align}

\begin{algorithm}[]
\caption{Numerical Exterior Derivative in $\R^n$}\label{alg:R3}
\begin{algorithmic}[1]
\Require \texttt{sampler} for $(k-1)$--form $\omega$, point $x\in \R^n$, and step size $\varepsilon > 0$
\Ensure Numerical values of $D\omega$ at $x$

\State Initialise standard basis vectors $e_1, \dots, e_n \in \mathbb{R}^n$
\State Initialise $D\omega$ as the zero tensor with shape $n\times \overset{k \text{ times}}{\cdots}\times n$ 
\ForAll{multi-indices \( (q_1, \dots, q_{k}) \) with \( 1 \le q_1 < \dots < q_{k} \le n \)}
\State ${\small\eD\omega_x(e_{q_1},\ldots,e_{q_k})\leftarrow\frac{1}{(2\varepsilon)}\sum\limits_{i=1}^k\sum\limits_{j=0}^1(-1)^{i+j}{\scriptstyle\texttt{sampler}}(x+(-1)^{j+1} \varepsilon e_{q_i}){[q_1,\dots, \widehat{q_i} ,\dots,q_k]}}$
\EndFor
\State \Return Anti-symmetrisation of $\eD\omega$
\end{algorithmic}
\end{algorithm}

In Algorithm~\ref{alg:R3} we propose a computational routine for approximating \eqref{eq:numerical_approx}. Examples~\ref{ex:1R3} and~\ref{ex:2R3} illustrate its execution in $\R^3$.

\begin{example}
\label{ex:1R3}
    Take $x = (1,1,1)\in \R^3$, and consider a $1$-form $\omega \in \Omega^1(\R^3)$ sampled within an $\varepsilon=0.01$-neighbourhood of $x$ using a black-box sampler as the one described in Algorithm~\ref{alg:sampler}, which produced the data cloud in Table~\ref{tab:datacloud}. Actually, this is a sample for the $1$-form whose analytical expression is $\omega = xdx + ydy + zdz$, i.e. $\omega$ is the dual $1$-form of the radial vector field in $\mathbb{R}^3$; if one knew that a priori, one would know immediately that $d\omega = 0$. Below we will see that our method indeed gives $D\omega_x \approx 0$.

    \begin{table}[]
        \centering
        \begin{tabular}{c|c}
            \textbf{Point} $\mathbf{p}$ &  $\mathbf{\omega_p}$\\
            \hline
        $x+\varepsilon e_1= (1.01, 1,1) $ & $(1.01,1,1)$\\
        $x-\varepsilon e_1= (0.99,1,1) $  & $(0.99,1,1)$\\
        $x+\varepsilon e_2= (1,1.01,1) $ 
        & $(1,1.01,1)$\\
        $x-\varepsilon e_2= (1,0.99,1) $ 
        & $(1,0.99,1)$\\
        $x+\varepsilon e_3= (1,1,1.01) $ 
        & $(1,1,1.01)$\\
        $x-\varepsilon e_3= (1,1,0.99) $ 
        & $(1,1,0.99)$\\

        \end{tabular}
        \caption{Samples of $\omega$ on an $\varepsilon=0.01$-neighbourhood of $x$.}
        \label{tab:datacloud}
    \end{table}

    The components of the numerical exterior derivative of $\omega$ at $x$ can then be approximated by 
    $$ 
    \eD\omega_x(e_i, e_j) \approx \frac{1}{(0.02)}
    (\omega_{x-\varepsilon e_j}[i]+
     \omega_{x+\varepsilon e_i}[j]-
      \omega_{x+\varepsilon e_j}[i]-
       \omega_{x-\varepsilon e_i}[j]),
    $$ where $\omega_p[i]$ denotes the $i$-th position in the $\omega_p$ array. Thus, as expected,
    \begin{align*}
    \eD\omega_x &\approx \frac{1}{0.02}\left((1+1-1-1)e^1\wedge e^2 + (1+1-1-1) e^1\wedge e^3 + (1+1-1-1) e^2\wedge e^3\right)\\
    &= 0e^1\wedge e^2 + 0 e^1\wedge e^3 +0 e^2\wedge e^3=0.
    \end{align*}
\end{example}

\begin{example}
    \label{ex:2R3}
    Let $x' = (1,2,3)\in \R^3$, $\omega' \in \Omega^2(\R^3)$, $\varepsilon=0.01$. The corresponding data cloud for $\omega$ is shown in Table~\ref{tab:datacloud2}. This is a sample for the $2$-form whose analytical expression is $\omega' = xdy \wedge dz$. However, as before, we rely solely on the numerical samples rather than the explicit analytical expression.
  
    \begin{table}[]
        \centering
        \begin{tabular}{c|c|c|c|c|c|c}
            \textbf{Point} $\mathbf{p}$ &  \makecell{$x+\varepsilon e_1=$\\$(1.01, 2,3)$} &
            \makecell{$x-\varepsilon e_1=$\\$(0.99,2,3)$} &
            \makecell{$x+\varepsilon e_2=$\\$(1,2.01,3)$} &
            \makecell{$x-\varepsilon e_2=$\\$(1,1.99,3)$} &
            \makecell{$x+\varepsilon e_3=$\\$(1,2,3.01)$} &
            \makecell{$x-\varepsilon e_3=$\\$(1,2,2.99)$} \\
            
            \hline
            \rule{0pt}{8mm}
            $\mathbf{\omega'_p}$ &
            \makecell{\scalebox{0.8}{$\begin{bmatrix}
            0 & 0 & 0\\
            0 & 0 & 1.01 \\
            0 & -1.01 & 0 
            \end{bmatrix}$}} &
            \makecell{\scalebox{0.8}{$\begin{bmatrix}
            0 & 0 & 0\\
            0 & 0 & 0.99 \\
            0 & -0.99 & 0 
            \end{bmatrix}$}}&
            \makecell{\scalebox{0.8}{$\begin{bmatrix}
            0 & 0 & 0\\
            0 & 0 & 1 \\
            0 & -1 & 0 
            \end{bmatrix}$}}&
            \makecell{\scalebox{0.8}{$\begin{bmatrix}
            0 & 0 & 0\\
            0 & 0 & 1 \\
            0 & -1 & 0 
            \end{bmatrix}$}}&
            \makecell{\scalebox{0.8}{$\begin{bmatrix}
            0 & 0 & 0\\
            0 & 0 & 1 \\
            0 & -1 & 0  
            \end{bmatrix}$}}&
            \makecell{\scalebox{0.8}{$\begin{bmatrix}
            0 & 0 & 0\\
            0 & 0 & 1 \\
            0 & -1 & 0 
            \end{bmatrix}$}}

        \end{tabular}
        \caption{Samples of $\omega'$ on an $\varepsilon=0.01$-neighbourhood of $x$.}
        \label{tab:datacloud2}
    \end{table}

    The components of the numerical exterior derivative of $\omega$ at $x$ can then be approximated by 
    $$
    \eD\omega'_x(e_i, e_j, e_k) \approx \frac{1}{(0.02)}
    (\omega'_{x+\varepsilon e_i}[j,k]
     -\omega'_{x-\varepsilon e_i}[j,k]
     -\omega'_{x+\varepsilon e_j}[i,k]
     +\omega'_{x-\varepsilon e_j}[i,k]
     +\omega'_{x+\varepsilon e_k}[i,j]
     -\omega'_{x-\varepsilon e_k}[i,j]
     ),
    $$
    where $\omega'_p[i,j]$ denotes the $(i,j)$-th position in the $\omega'_p$ multi-array. As expected,
    \begin{align*}
    \eD\omega'_x &\approx \frac{1}{0.02}\left((1.01-0.99-0+0+0-0)e^1\wedge e^2 \wedge e^3\right)\\
    &= 1 e^1\wedge e^2\wedge e^3.
    \end{align*}
\end{example}

\bmhead{Acknowledgements} The main ideas developed in \S\S\ref{sec:trisection lemma}--\ref{sec:Stokes_Theorem} were originally
proposed by Felipe Acker
in \cite{acker1,paper,paper2}, where they are proved in somewhat flexible terms; above all, the most comprehensive source \cite{acker1} is only available in Portuguese. So we dedicate this article to Prof. Acker and hope to bring his innovation, now in a fully rigorous form, to a much broader audience.

HSE was supported by the São Paulo Research Foundation (Fapesp) [2021/04065-6] \textit{BRIDGES collaboration} and the Brazilian National Council for Scientific
and Technological Development (CNPq) [311128/2020-3]. TSRS was supported by Fapesp grant [2022/09891-4].













\bibliography{ref}%

@book{Federer1996,
  title = {Geometric measure theory},
  ISBN = {9783642620102},
  ISSN = {1431-0821},
address={Berlin Heidelberg},
edition={},
  url = {http://dx.doi.org/10.1007/978-3-642-62010-2},
  DOI = {10.1007/978-3-642-62010-2},
  journal = {Classics in Mathematics},
  publisher = {Springer},
  author = {Federer,  Herbert},
  year = {1996}
}

@misc{sagemath,
    author={{The Sage Developers} and William Stein and David Joyner and David Kohel and John Cremona and Eröcal, Burçin},
    title={{SageMath, version 9.0}},
    year={2020},
    url={http://www.sagemath.org},
}

@book{FEEC,
author = {Arnold, Douglas N.},
title = {Finite element exterior calculus},
publisher = {Society for Industrial and Applied Mathematics},
year = {2018},
doi = {10.1137/1.9781611975543},
address = {Philadelphia, PA},
edition   = {},
URL = {https://epubs.siam.org/doi/abs/10.1137/1.9781611975543},
eprint = {https://epubs.siam.org/doi/pdf/10.1137/1.9781611975543}
}

@phdthesis{Hirani,
  doi = {10.7907/ZHY8-V329},
  url = {https://resolver.caltech.edu/CaltechETD:etd-05202003-095403},
  author = {Hirani,  Anil Nirmal},
  language = {en},
  school={California Institute of Technology},
  title = {Discrete exterior calculus},
  publisher = {California Institute of Technology},
  year = {2003},
  copyright = {No commercial reproduction,  distribution,  display or performance rights in this work are provided.}
}

@book{acker1,
  author       = {Acker, F.},
  title        = {Análise vetorial clássica},
  year         = {2012},
  publisher    = {Sociedade Brasileira de Matemática},
  address      = {Rio de Janeiro},
  language     = {portuguese}
}

@article{paper,
  author       = {Acker, F.},
  title        = {The missing link},
  journal      = {The Mathematical Intelligencer},
  year         = {1996},
  volume       = {18},
  number       = {3},
  pages        = {4--9},
  publisher    = {Springer}
}

@techreport{paper2,
  author       = {Acker, F.},
  title        = {The mean value theorem for differential forms},
  institution  = {Instituto de Matemática, Universidade Federal do Rio de Janeiro (UFRJ)},
  year         = {2004},
  type         = {Technical report}
}

@book{spivak,
  author       = {Spivak, Michael},
  title        = {Calculus on manifolds},
  year         = {1965},
  publisher    = {Addison-Wesley},
  address      = {Reading, MA}
}

@book{Arnold1989,
  author       = {Arnold, Vladimir I.},
  title        = {Mathematical methods of classical mechanics},
  series       = {Graduate Texts in Mathematics},
  volume       = {60},
  year         = {1989},
  publisher    = {Springer},
  address      = {New York}
}

\end{document}